\documentclass[a4paper,twoside,onecolumn,sn-mathphys-num]{sn-jnl}
\usepackage{graphicx}
\usepackage{multirow}
\usepackage{amsmath,amssymb,amsfonts}
\usepackage{amsthm}
\usepackage{mathrsfs}
\usepackage[title]{appendix}
\usepackage{xcolor}
\usepackage{textcomp}
\usepackage{manyfoot}
\usepackage{booktabs}
\usepackage{algorithm}
\usepackage{algorithmicx}
\usepackage{algpseudocode}
\usepackage{listings}
\usepackage{tikz-cd}
\usepackage{bbm}
\usepackage[T1]{fontenc}
\definecolor{rosso}{rgb}{0.85,0,0}

\theoremstyle{thmstyleone}
\newtheorem{theorem}{Theorem}
\theoremstyle{thmstyletwo}

\newtheorem{remark}{Remark}
\theoremstyle{thmstylethree}

\newcommand{\beq}{\begin{equation}}
\newcommand{\eeq}{\end{equation}}
\newcommand{\beqa}{\begin{eqnarray}}
\newcommand{\eeqa}{\end{eqnarray}}
\newcommand {\rmd} { {\mathrm d} }
\usepackage{cases}
\begin{document}

\title[Spatial segregation across invading fronts in models for the growth of heterogeneous cell populations]{Spatial segregation across travelling fronts in individual-based and continuum models for the growth of heterogeneous cell populations}

\author*[1]{\fnm{Jos\'{e}} A. \sur{Carrillo}}\email{carrillo@maths.ox.ac.uk}

\author[2]{\fnm{Tommaso} \sur{Lorenzi}}

\author[3,4]{\fnm{Fiona} R. \sur{Macfarlane}}

\affil[1]{\orgdiv{Mathematical Institute}, \orgname{University of Oxford}, \orgaddress{\city{Oxford}, \country{United Kingdom}}}

\affil[2]{\orgdiv{Department of Mathematical Sciences ``G. L. Lagrange''}, \orgname{Politecnico di Torino}, \orgaddress{\city{Turin}, \country{Italy}}}

\affil[3]{\orgdiv{School of Mathematics and Statistics}, \orgname{University of St Andrews}, \orgaddress{\city{St Andrews}, \country{United Kingdom}}}

\affil[4]{\orgdiv{Applied BioSimulation}, \orgname{Certara (UK)}, \orgaddress{\city{Sheffield}, \country{United Kingdom}}}

\abstract{We consider a partial differential equation model for the growth of heterogeneous cell populations subdivided into multiple distinct discrete phenotypes. In this model, cells preferentially move towards regions where they are less compressed, and thus their movement occurs down the gradient of the cellular pressure. The cellular pressure is defined as a weighted sum of the densities (i.e. the volume fractions) of cells with different phenotypes. To translate into mathematical terms the idea that cells with distinct phenotypes have different morphological and mechanical properties, both the cell mobility and the weighted amount the cells contribute to the cellular pressure vary with their phenotype. We formally derive this model as the continuum limit of an on-lattice individual-based model, where cells are represented as single agents undergoing a branching biased random walk corresponding to  phenotype-dependent and pressure-regulated cell division, death, and movement. Then, we study travelling wave solutions whereby cells with different phenotypes are spatially segregated across the invading front. Finally, we report on numerical simulations of the two models, demonstrating excellent agreement between them and the travelling wave analysis. The results presented here indicate that inter-cellular variability in mobility can \textcolor{black}{support the maintenance of spatial segregation across invading fronts, whereby cells with a higher mobility drive invasion by occupying regions closer to the front edge.}}


\maketitle

\section{Introduction}\label{sec:section1}

\subsection{Background}\label{sec:section1p1}
Systems of partial differential equations (PDEs) modelling the growth of populations that disperse to avoid crowding~\cite{bertsch1985interacting,gurtin1984note} have been 
applied to describe the spatiotemporal dynamics of multiple types of cells underpinning tissue development, wound healing, and tumour
growth~\cite{ambrosi2002closure,bertsch2012nonlinear,bubba2020hele,byrne2003modelling,carrillo2019population,chaplain2006mathematical,ciarletta2011radial,david2024degenerate,drasdo2012modeling,giverso2022effective,gwiazda2019two,lorenzi2017interfaces,mimura2010free,oelschlager1989derivation,roose2007mathematical,preziosi2009multiphase}. 

Focusing on a one-dimensional spatial scenario, and considering two cell types, a prototypical example of these models is provided by the following PDE system
\begin{subnumcases}{\label{eq:PDEmodelORI}}
   \partial_t n_1 - \partial_x \left(n_1 \, \partial_x \rho \right) = G_1(\rho) \, n_1, & \label{eq:PDEmodelORI_1}
   \\
   \nonumber\\
   \partial_t n_2 - \partial_x \left(n_2 \, \partial_x \rho \right) = G_2(\rho) \, n_2, & $(t,x) \in (0,\infty) \times \mathbb{R}$ \label{eq:PDEmodelORI_2}\\
   \nonumber\\
   \rho(t,x) := n_1(t,x) + n_2(t,x),& \label{eq:PDEmodelORI_3}
\end{subnumcases}
Here, the functions $n_1(t,x)$ and $n_2(t,x)$ are the densities (i.e. the volume fractions) of cells of types 1 and 2 at position $x$ at time $t$, and the function $\rho(t,x)$ defined via~\eqref{eq:PDEmodelORI_3} is the total cell density (i.e. the total cell volume fraction). The transport terms on the left-hand sides of the PDEs~\eqref{eq:PDEmodelORI_1} and \eqref{eq:PDEmodelORI_2} model the effect of cell movement and capture the tendency of cells to move away from overcrowded regions (i.e. to move down the gradient of the total cell density)~\cite{chaplain2006mathematical}. Moreover, the reaction terms on the right-hand sides model the effect of cell division and death, with the functions $G_1(\rho)$ and $G_2(\rho)$ being the net growth rates of the densities of cells of types 1 and 2. These functions depend on the total cell density so as to integrate the effect of density-dependent inhibition of growth (i.e. the cessation of growth at sufficiently high cell density)~\cite{lieberman1981density}.

A variation on the model~\eqref{eq:PDEmodelORI} is given by the following PDE system
\begin{subnumcases}{\label{eq:PDEmodelORIp}}
   \partial_t n_1 - \partial_x \left(n_1 \, \partial_x p \right) = G_1(p) \, n_1, & \label{eq:PDEmodelORIp_1}
   \\
   \nonumber\\
   \partial_t n_2 - \partial_x \left(n_2 \, \partial_x p \right) = G_2(p) \, n_2, & $(t,x) \in (0,\infty) \times \mathbb{R}$ \label{eq:PDEmodelORIp_2}\\
   \nonumber\\
   p(t,x) := \Pi[\rho](t,x), \  \rho(t,x) := n_1(t,x) + n_2(t,x),& \label{eq:PDEmodelORIp_3}
\end{subnumcases}
The PDE system~\eqref{eq:PDEmodelORIp} can be obtained by replacing the cell density $\rho(t,x)$ in~\eqref{eq:PDEmodelORI} with the cellular pressure $p(t,x)$ and then closing the resulting system for the cell density functions by defining the cell pressure as a function of the total cell density through an appropriate constitutive relation $\Pi[\rho]$~\cite{ambrosi2002closure,perthame2014hele}. This makes it possible to incorporate into the model the effects of pressure-regulated growth -- i.e. the fact that cells will stop dividing if the pressure at their current position overcomes a critical value~\cite{byrne2003modelling,byrne2009individual,drasdo2012modeling,ranft2010fluidization} -- and pressure-regulated cell movement -- i.e. the movement of cells down the gradient of the cellular pressure towards regions where they are less compressed~\cite{byrne1997free,byrne2003modelling}.   

These models, which provide a population-level description of cell dynamics, can be derived as the continuum limits of underlying individual-based models, which track the dynamics of single cells and are thus able to capture the finer details of single-cell movement, division, and death~\cite{drasdo2005coarse}. To this end, a range of limiting procedures have been developed and employed for transitioning between individual-based models for the growth of cell populations and continuum models formulated as systems of PDEs of the form of~\eqref{eq:PDEmodelORI} and~\eqref{eq:PDEmodelORIp} and related forms~\cite{alasio2022existence,carrillo2019population,chaplain2020bridging,dyson2012macroscopic,fozard2010continuum,galiano2014cross,lorenzi2020individual,penington2011building,pillay2017modeling,pillay2018impact,simpson2011models,simpson2024discrete}.

An interesting feature of PDE systems like~\eqref{eq:PDEmodelORI} and~\eqref{eq:PDEmodelORIp} is that they \textcolor{black}{can support} spatial segregation between different cell types -- i.e. the fact that if cells of different types are initially separated (i.e. they occupy distinct regions of the spatial domain at the initial time) then they will remain separated also at later times. Mathematically speaking, this means that the solutions of these PDE systems may exhibit a propagation of segregation property~\cite{david2024degenerate}. This is an interesting mathematical aspect, the study of which has received increasing attention in recent decades~\cite{bertsch2015travelling,bertsch2012nonlinear,bertsch1987interacting,bertsch1987degenerate,bertsch1985interacting,burger2019segregation,carrillo2018zoology,carrillo2018splitting,girardin2022spatial,galiano2014cross,lorenzi2017interfaces,mimura2010free,FBC24}. Furthermore, it makes such models an appropriate theoretical framework to investigate the mechanisms underlying cell segregation processes leading to the formation of sharp borders between cells of distinct types or with different phenotypes, which is observed both in normal development and in tumourigenesis~\cite{batlle2012molecular}.

\subsection{Object of study}\label{sec:section1p2}
In this paper, we consider the following model for the growth of a phenotypically heterogeneous population comprising cells of different types (i.e. with different phenotypes), which are labelled by the index $i=1, \ldots, I$ with $I  \textcolor{black}{\geq} 2$:
\begin{subnumcases}{\label{eq:PDEmodel}}
   \partial_t n_1 - \mu_1 \, \partial_x \left(n_1 \, \partial_x p \right) = G_1(p) \, n_1 & \label{eq:PDEmodel_1}
   \\
   \nonumber\\
   \partial_t n_i - \mu_i \, \partial_x \left(n_i \, \partial_x p \right) = 0, \quad i=2, \ldots, I & $(t,x) \in (0,\infty) \times \mathbb{R}$ \label{eq:PDEmodel_2}\\
   \nonumber\\
   p(t,x) := \sum_{i=1}^I \omega_i \, n_i(t,x),& \label{eq:PDEmodel_3}
\end{subnumcases}

In the model~\eqref{eq:PDEmodel}, the functions $n_1(t,x), \ldots, n_I(t,x)$ represent the densities (i.e. the volume fractions) of cells with phenotypes $1, \ldots, I$ at position $x$ at time $t$, while the function $p(t,x)$ represents the cellular pressure, which is defined as a function of the cell densities through the constitutive relation~\eqref{eq:PDEmodel_3}. The positive parameters $\omega_1, \ldots, \omega_I$ provide a measure of the weighted amounts that cells with phenotypes $1, \ldots, I$ contribute towards the cellular pressure, and the values of these parameters are related to the morphological and mechanical properties of the cells (such as cell size and stiffness), which may vary depending on the cell phenotype~\cite{masaeli2016multiparameter}. In analogy with model~\eqref{eq:PDEmodelORIp}, the second terms on the left-hand sides of the PDEs~\eqref{eq:PDEmodel_1} and~\eqref{eq:PDEmodel_2} are the rates of change of the densities of cells with phenotypes $1, \ldots, I$ due to pressure-regulated cell movement. The positive parameter $\mu_i$ is the mobility coefficient of cells with phenotype $i$~\cite{ambrosi2002closure,byrne2009individual}, which in model~\eqref{eq:PDEmodelORIp} is implicitly assumed to be the same for all cells (and it is then set to 1), but in fact it can vary due again to differences in morphological and mechanical properties (such as cell elongation and nucleus deformability) between cells with different phenotypes~\cite{kalukula2022mechanics,lamouille2014molecular}. Moreover, similarly to model~\eqref{eq:PDEmodelORIp}, the term on the right-hand side of the PDE~\eqref{eq:PDEmodel_1} is the rate of change of the density of cells with phenotype $1$ due to pressure-regulated growth and, therefore, the function $G_1(p)$ is the net growth rate of the density of cells with phenotype $1$ when exposed to the cellular pressure $p$. Focusing on the impact of inter-cellular variability in mobility on cell dynamics rather than variability in division and death rates, the right-hand sides of the PDEs~\eqref{eq:PDEmodel_2} are set to zero. This corresponds to a biological scenario where division and death of cells with phenotypes $2, \ldots, I$ occur on much slower time scales compared to division and death of cells with phenotype $1$, and can thus be neglected. For example, in the context of cancer, tumour-derived leader cells undergo epithelial to mesenchymal transition acquiring mesenchymal characteristics including significantly increased motility and significantly reduced proliferation~\cite{chen2020makes,huang2022molecular,konen2017image,vilchez2021decoding,wang2023tumour,zanotelli2021mechanoresponsive}. Specifically, it has been observed that on short time-scales the proliferation rate of leader cells compared to follower cells is negligible~\cite{konen2017image}. Hence, in the framework of model~\eqref{eq:PDEmodel}, cells with phenotype $1$ could be regarded as follower-type cells, while cells with phenotypes $2, \ldots, I$ could be regarded as leader-type cells with heterogeneous morphological and mechanical properties.

In contrast to previous works on related models, which considered two cell phenotypes only (i.e. $i=1,2$) and assumed the values of the weights $\omega_i$ and the values of the mobility coefficients $\mu_i$ to be the same for both phenotypes, in this paper we consider an arbitrary number $I  \textcolor{black}{\geq} 2$ of cell phenotypes and allow both the values of $\omega_i$ and the values of $\mu_i$ to vary with the phenotype, in order to capture intra-population phenotypic heterogeneity more accurately.

\subsection{Outline of the paper}\label{sec:section1p3}
Building on the modelling framework developed in~\cite{chaplain2020bridging}, we first formulate an individual-based model (see Section~\ref{sec:section2}) where cells are represented as single agents undergoing phenotype-dependent and pressure-regulated cell division, death, and movement according to a set of rules. These rules result in cells performing a branching biased random walk over the one-dimensional lattice that represents the spatial domain~\cite{hughes1996random,johnston2012mean,penington2011building}. Then, using a limiting procedure analogous to the one that we employed in~\cite{bubba2020discrete,macfarlane2020hybrid,macfarlane2022individual}, we formally derive model~\eqref{eq:PDEmodel} as the continuum limit of this individual-based model (see Section~\ref{sec:section3} and Appendix~\ref{app:appendixA}). After that, generalising the method of proof that developed in~\cite{lorenzi2017interfaces}, we carry out travelling wave analysis of the model~\eqref{eq:PDEmodel} (see Section~\ref{sec:section4}) and study, under appropriate assumptions on the function $G_1$ and the mobility coefficients $\mu_1, \ldots, \mu_I$, travelling front solutions wherein cells with phenotypes labelled by different values of the index $i=1, \ldots, I$ are spatially segregated across the front (i.e. they occupy distinct regions of the front). Finally, we report on numerical simulations of the individual-based model and numerical solutions of the PDE model~\eqref{eq:PDEmodel}, which demonstrate excellent agreement between the two models, thus validating the formal limiting procedure employed to derive the continuum limit of the individual-based model, and confirm the results of travelling wave analysis (see Section~\ref{sec:section5}). We conclude with a discussion of the main results obtained and provide a brief overview of possible research perspectives (see Section~\ref{sec:section6}).

\section{Individual-based and continuum models}\label{sec:section2}

\subsection{Individual-based model}
\label{sec:ibmodn}
Considering a one-dimensional spatial scenario, we let cells be distributed and move along the real line $\mathbb{R}$. We introduce the notation $\mathbb{R}_{+} := \{ z \in \mathbb{R} : z \geq 0 \}$, $\mathbb{R}^*_{+} :=  \mathbb{R}_{+} \setminus \{0\}$, and $\mathbb{N}_0 := \mathbb{N} \cup \{ 0 \}$, and then discretise the time variable $t\in \mathbb{R}_{+}$ and the space variable $x\in \mathbb{R}$ as $t_k=k\tau$ with $k\in\mathbb{N}_0$ and $x_j=j\Delta_x$ with $j\in\mathbb{Z}$, respectively, where $\tau \in \mathbb{R}^*_{+}$ represents the time-step and $\Delta_x \in \mathbb{R}^*_{+}$ represents the space-step. 

We consider a population comprising cells expressing one amongst $I  \textcolor{black}{\geq} 2$ distinct discrete phenotypes, each labelled by an index $i = 1,\ldots,I$, and define the density of cells with phenotype $i$ at position $x_j$ at time $t_k$, denoted $n_{i, j}^k$, as
\begin{equation}
n_{i, j}^k := \dfrac{N_{i, j}^k}{\Delta_x}, \label{eq:density_definition}
\end{equation} 
where $N_{i, j}^k$ is the number of cells with phenotype $i$ at position $x_j$ at time $t_k$. Moreover, we define the cellular pressure at position $x_j$ at time $t_k$, denoted $p_j^k$, as a function of $n_{i, j}^k$ through the following constitutive relation:
\begin{equation}
p_j^k:=\displaystyle{\sum_{i=1}^{I} \omega_i \ n_{i, j}^k}, \quad \text{with}\quad \omega_i\in\mathbb{R}_+^*.\label{eq:pressure_definition}
\end{equation} 
As mentioned in Section~\ref{sec:section1}, the constitutive relation~\eqref{eq:pressure_definition} translates into mathematical terms the idea that cells with each phenotype $i$ may contribute a different weighted amount, which is represented by the parameter $\omega_i$, to the cellular pressure. The value of the parameter $\omega_i$ is related to the morphological and mechanical properties of the cells; for instance, higher values of $\omega_i$ may correspond to larger cell size and/or larger cell stiffness~\cite{masaeli2016multiparameter}. The dynamics of the cells are governed by the rules summarised in Figure~\ref{fig:Fig1} and detailed in the following subsections.

\begin{figure}
\centering
\includegraphics[width=1\textwidth]{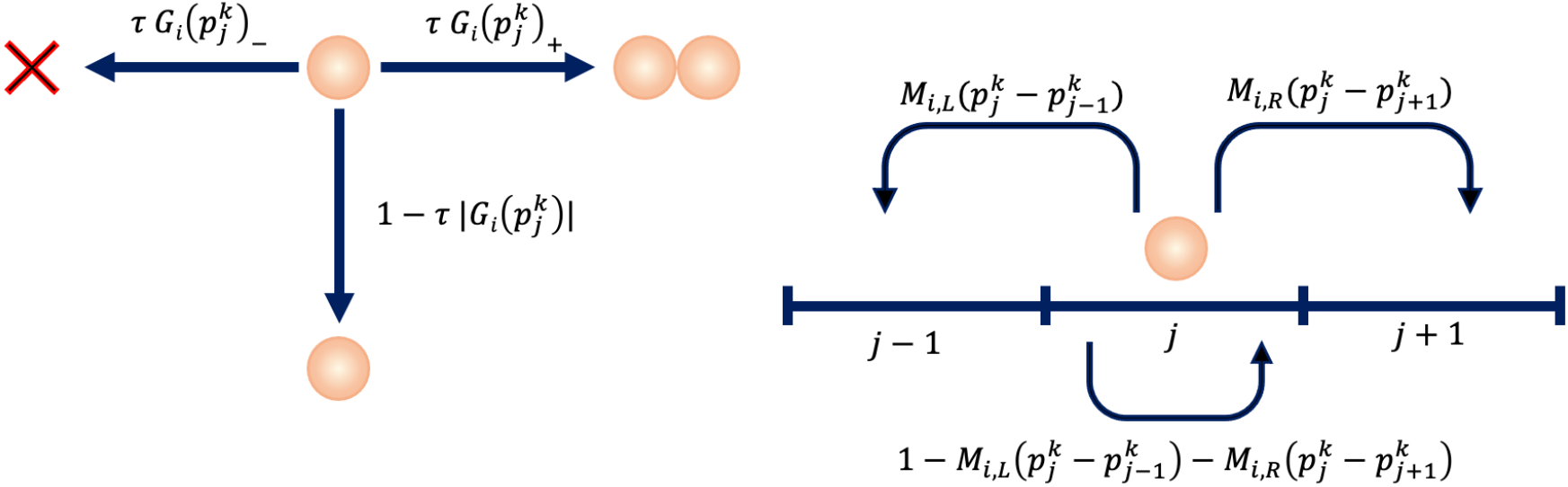}
\caption{{\bf Schematic overview of the mechanisms incorporated in the individual-based model.} Between time-steps $k$ and $k+1$, each cell with phenotype $i = 1,\ldots,I$ at spatial position $x_j$ may: divide with probability $\tau \, G_i(p_j^k)_+$, die with probability $\tau \, G_i(p_j^k)_-$, and remain quiescent with probability $1-\tau \, \left(G_i(p_j^k)_+ + G_i(p_j^k)_- \right) = 1 - \tau \, |G_i(p_j^k)|$ {({\bf \textit{left panel}})}; move to spatial positions $x_{j-1}$ or $x_{j+1}$ with probabilities $M_{i,L}(p_j^k-p_{j-1}^k)$ or $M_{i,R}(p_j^k-p_{j+1}^k)$, respectively, or remain stationary with probability $1-M_{i,L}(p_j^k-p_{j-1}^k)-M_{i,R}(p_j^k-p_{j+1}^k)$ {({\bf \textit{right panel}})}}
\label{fig:Fig1}
\end{figure}

\subsubsection{Modelling cell division and death}\label{sec:section2p1}
We model pressure-regulated cell division and death as a branching process along the spatial dimension, whereby cells divide and die with probabilities that depend on both their phenotype and the pressure they experience, as illustrated by the schematic in the left panel of Figure~\ref{fig:Fig1}. If cell division occurs, a dividing cell is instantly replaced by two identical progeny cells that inherit the spatial position and phenotype of the parent cell. Conversely, a cell undergoing cell death is instantly removed from the population. 

We introduce the function $G_i(p_j^k)$, which represents the net growth rate of the density of cells with phenotype $i=1,\ldots,I$ at spatial position $x_j$ at time $t_k$, and use this function to define the probability of cell division and death in the individual-based model. Specifically, between time-steps $k$ and $k+1$ we let a cell with phenotype $i$ at position $x_j$ divide with probability 
$$
\tau G_i(p_j^k)_+, \; \text{ where } \; (\cdot)_+=\max(0,\cdot),
$$ 
die with probability 
$$
\tau G_i(p_j^k)_-, \; \text{ where } \; (\cdot)_-=-\min(0,\cdot),
$$
and remain quiescent with probability 
$$
1-\tau \left(G_i(p_j^k)_+ + G_i(p_j^k)_- \right) = 1 - \tau |G_i(p_j^k)|.
$$ 
By choosing the time-step $\tau$ sufficiently small, we ensure that all the quantities above are between 0 and 1. 

In order to take into account the fact that cells will stop dividing if the pressure at their current position exceeds a critical value, known as the homeostatic pressure~\cite{basan2009homeostatic}, which is modelled by the parameter $\overline{p} \in \mathbb{R}_+^*$, and considering a scenario in which cells with phenotypes labelled by different values of the index $i$ may undergo cell division and death over different time scales, we make the following assumptions:
\begin{equation}
\label{eq:G_conditions}
G_i(p) := \alpha_i \, G(p), \quad \textcolor{black}{G(0) < \infty}, \quad G(\overline{p})=0, \quad \frac{\mathrm{d}G}{\mathrm{d}p}<0 \, \forall p\in\mathbb{R}_{+}.
\end{equation} 
Here the parameter $\alpha_i \in \mathbb{R}_{+}$ is linked to the time scale over which cells with phenotype $i$ undergo cell division and death. In particular, focusing on the case where division and death of cells with phenotypes $i=2, \ldots, I$ can be neglected, since they occur on much slower time scales compared to division and death of cells with phenotype $i=1$, we also assume
\beq
\label{eq:alpha_conditions}
\alpha_1 > 0 \;\; \text{ and } \;\; \alpha_i = 0 \; \text{ for } \; i=2, \ldots, I.
\eeq

\subsubsection{Modelling cell movement}\label{sec:section2p2}
We model directional cell movement in response to pressure differentials as a biased random walk, whereby the movement probabilities depend on the difference between the cellular pressure at the position occupied by the cell and the cellular pressure at neighbouring positions, as illustrated by the schematic in the right panel of Figure~\ref{fig:Fig1}. In particular, we assume that cells move down the gradient of the pressure so as to reach regions where they are less compressed. Moreover, in order to capture the fact that the phenotype of the cells determines their sensitivity to the pressure gradient, and thus their mobility, we introduce the parameters $\gamma_i\in \mathbb{R}_+^*$ to model the sensitivity to the pressure gradient of cells with phenotypes $i=1,\ldots,I$.

In the individual-based model, between time-steps $k$ and $k+1$ we let a cell with phenotype $i$ at position $x_j$ move to position $x_{j-1}$ (i.e. move left) with probability 
$$
{M}_{i,L}(p_j^k-p_{j-1}^k),
$$
move to position $x_{j+1}$ (i.e. move right) with probability 
$$
{M}_{i,R}(p_j^k-p_{j+1}^k),
$$
and remain stationary with probability 
$$
1-\left({M}_{i,L}(p_j^k-p_{j-1}^k)+{M}_{i,R}(p_j^k-p_{j+1}^k)\right).
$$
Specifically, recalling that, as described in Section~\ref{sec:section2p1}, the parameter $\overline{p}$ represents the homeostatic pressure, we use the following definitions  
\begin{equation}
\begin{array}{l}
 \quad {M}_{i,L}(p_j^k-p_{j-1}^k)=\displaystyle{\gamma_i \frac{\left(p_j^k-p_{j-1}^k\right)_+}{2\overline{p}}},\\\ \\
 \quad {M}_{i,R}(p_j^k-p_{j+1}^k)=\displaystyle{\gamma_i \frac{\left(p_j^k-p_{j+1}^k\right)_+}{2\overline{p}}}, \end{array}\quad \text{where} \quad (\cdot)_+=\max(0,\cdot),
\label{eq:define_movementIB}
\end{equation}
and choose the model parameters and functions such that $0 \leq {M}_{i,L}(p_j^k-p_{j-1}^k)+{M}_{i,R}(p_j^k-p_{j+1}^k)<1$ for all $i$, $j$, and $k$. 

Without loss of generality, considering a scenario where cells with phenotypes labelled by smaller values of the index $i$ have a lower sensitivity to the pressure gradient, and thus a lower mobility, we assume
\begin{equation}
0<\gamma_1<\gamma_2<\ldots<\gamma_{I-1}<\gamma_I.
\label{eq:gamma_conditions}
\end{equation}

\begin{remark}
\label{rem:rem1}
Taken together, assumptions~\eqref{eq:G_conditions}-\eqref{eq:alpha_conditions} and~\eqref{eq:gamma_conditions} correspond to the situation in which, due to a trade-off between cell proliferative and migratory abilities, fast-dividing cells with phenotype $i=1$ display the lowest mobility~\cite{chen2020makes,huang2022molecular,konen2017image,vilchez2021decoding,wang2023tumour,zanotelli2021mechanoresponsive}.
\end{remark}

\subsection{Corresponding continuum model}\label{sec:section3}
\textcolor{black}{As detailed in Appendix~\ref{app:appendixA},} from the branching biased random walk underlying the individual-based model presented in the previous section, one formally derives, as the corresponding continuum limit, a PDE system for the functions $n_i(t,x)$, each modelling the density (i.e. the volume fraction) of cells with phenotype $i=1, \ldots, I$, at position $x \in \mathbb{R}$ at time $t\in \mathbb{R}_+$. This is done through an extension of the limiting procedure that we employed in \cite{bubba2020discrete,macfarlane2020hybrid,macfarlane2022individual} -- see also \cite{simpson2007simulating,johnston2012mean,simpson2011models,simpson2024discrete} for related strategies.

\textcolor{black}{
In summary, one writes down a balance equation for the density of cells with phenotype $i$ at spatial position $x_j$ at time $t_{k+1}$, which depends on cell densities at time $t_{k}$ at position $x_j$ and neighbouring positions $x_{j-1}$ and $x_{j+1}$, as a result of cell movement and cell division and death. Specifically, one has}

\textcolor{black}{\begin{eqnarray}
\label{eq:mastereq}
n_{i,\ j}^{k+1}&=&n_{i,\ j}^k\left\{ \left(1+\tau G_i(p_j^k)\right)\left[1-\frac{\gamma_i \left(p_j^k-p_{j+1}^k\right)_+}{2\overline{p}}-\frac{\gamma_i \left(p_j^k-p_{j-1}^k\right)_+}{2\overline{p}}\right]\right\}\nonumber\\
&&+n_{i,\ j+1}^k\left\{ \left(1+\tau G_i(p_{j+1}^k)\right)\left[\frac{\gamma_i \left(p_{j+1}^k-p_{j}^k\right)_+}{2\overline{p}}\right]\right\}\nonumber\\
&&+n_{i,\ j-1}^k\left\{ \left(1+\tau G_i(p_{j-1}^k)\right)\left[\frac{\gamma_i \left(p_{j-1}^k-p_{j}^k\right)_+}{2\overline{p}}\right]\right\}.
\end{eqnarray}}
\textcolor{black}{From the balance equation~\eqref{eq:mastereq}}, employing a formal limiting procedure that includes letting $\Delta_x\rightarrow 0$ and $\tau \rightarrow 0$ in such a way that
\beq
\label{ass:formderi}
\frac{\gamma_i \Delta_x^2}{2\tau \overline{p}}\rightarrow \mu_i, \quad \text{where} \quad \mu_i \in \mathbb{R}^*_+,\ i=1, \ldots, I,
\eeq
\textcolor{black}{under the constitutive relation~\eqref{eq:pressure_definition} one formally obtains the PDE system
\begin{equation}
\label{eq:PDEmodelrev}
\begin{cases}
\begin{array}{l}
\displaystyle{\partial_t n_i - \mu_i \, \partial_x \left(n_i \, \partial_x p \right) = G_i(p) \, n_i, \quad i=1, \ldots, I},
\\\\
\displaystyle{p(t,x) := \sum_{i=1}^I \omega_i \, n_i(t,x),}
\end{array}
\quad (t,x) \in (0,\infty) \times \mathbb{R},
\end{cases}
\end{equation}
subject to the following assumptions on the mobility coefficients, $\mu_i$, 
\begin{equation}
\label{ass:modpar3}
0 < \mu_1 < \mu_2 < \ldots < \mu_{I-1} < \mu_I.
\end{equation}
Assumptions~\eqref{ass:modpar3} descend from assumptions~\eqref{eq:gamma_conditions} when conditions~\eqref{ass:formderi} hold. Under the additional assumptions~\eqref{eq:G_conditions}-\eqref{eq:alpha_conditions} the PDE system~\eqref{eq:PDEmodelrev} then reduces to the PDE system~\eqref{eq:PDEmodel}.}

\section{Travelling wave analysis}\label{sec:section4}
 \textcolor{black}{In this section, under assumptions~\eqref{eq:G_conditions}-\eqref{eq:alpha_conditions} and~\eqref{ass:modpar3}, we carry out travelling wave analysis of the continuum model~\eqref{eq:PDEmodel}. 
Substituting the travelling wave ansatz 
$$
n_i(t,x) = n_i(z), \quad z := x-c \, t, \quad c \in \mathbb{R}^*_+,
$$
where $c$ is the travelling wave speed, into the model~\eqref{eq:PDEmodel} yields
\begin{subnumcases}{\label{eq:PDEmodelTW}}
  -c \, n_1' - \mu_1 \, \left(n_1 \, p' \right)' = G_1(p) \, n_1, & \label{eq:PDEmodelTW_1}
   \\
      \nonumber\\
   -c \, n_i' - \mu_i \, \left(n_i \, p' \right)' =  0, \quad i=2, \ldots, I, &  $z \in \mathbb{R}.$  \label{eq:PDEmodelTW_1a}\\
   p(z) := \sum_{i=1}^I \omega_i \, n_i(z), &  \label{eq:PDEmodelTW_2}
\end{subnumcases}
}
We seek travelling wave solutions that satisfy the following conditions
\beq
\label{ass:SuppTW1}
n_i(z) 
\begin{cases}
> 0, \;\; \text{for } z \in (z_{i-1}, z_i)
\\
=0 , \;\;  \text{for } z \not\in (z_{i-1}, z_i),
\end{cases}
\quad i=1, \ldots, I,
\eeq
where
\beq
\label{ass:SuppTW2}
-\infty =: z_0 < z_1 := 0 < z_2 < \ldots < z_{I-1} < z_I < \infty,
\eeq
along with the asymptotic (i.e. boundary) condition
\beq
\label{ass:asycon}
p(-\infty) = \overline{p}.
\eeq
Note that, under conditions~\eqref{ass:SuppTW1}-\eqref{ass:SuppTW2}, conservation of mass ensures that
\beq
\label{eq:masscon}
\int_{z_{i-1}}^{z_i} n_i(z) \, {\rm d}z = M_i, \quad M_i \in \mathbb{R}^*_+, \quad i=2, \ldots, I,
\eeq
where the parameter  $M_i$ represents the number (i.e. the total volume fraction) of cells with phenotype $i=2, \ldots, I$ in the population.\\ 

\begin{remark}
\label{rem:rem2}
Conditions~\eqref{ass:SuppTW1}-\eqref{ass:SuppTW2} correspond to a scenario in which cells with phenotypes labelled by different values of the index $i$ are spatially segregated across invading fronts, which are represented by travelling wave solutions of the continuum model~\eqref{eq:PDEmodel}, i.e. solutions of the system of differential equations~\eqref{eq:PDEmodelTW}. More precisely, also in the light of assumptions~\eqref{eq:G_conditions}-\eqref{eq:alpha_conditions} and~\eqref{ass:modpar3}, cells with phenotype $i=1$ (i.e. fast-dividing cells with the lowest mobility) make up the bulk of the population in the rear of the wave, while cells with phenotypes labelled by increasing values of $i>1$, which display a higher mobility, are found in the regions closer to the invading edge (cf. the schematic in Figure~\ref{fig:Fig2}).
\end{remark}

\begin{figure}[H]
\centering
\includegraphics[width=1\textwidth]{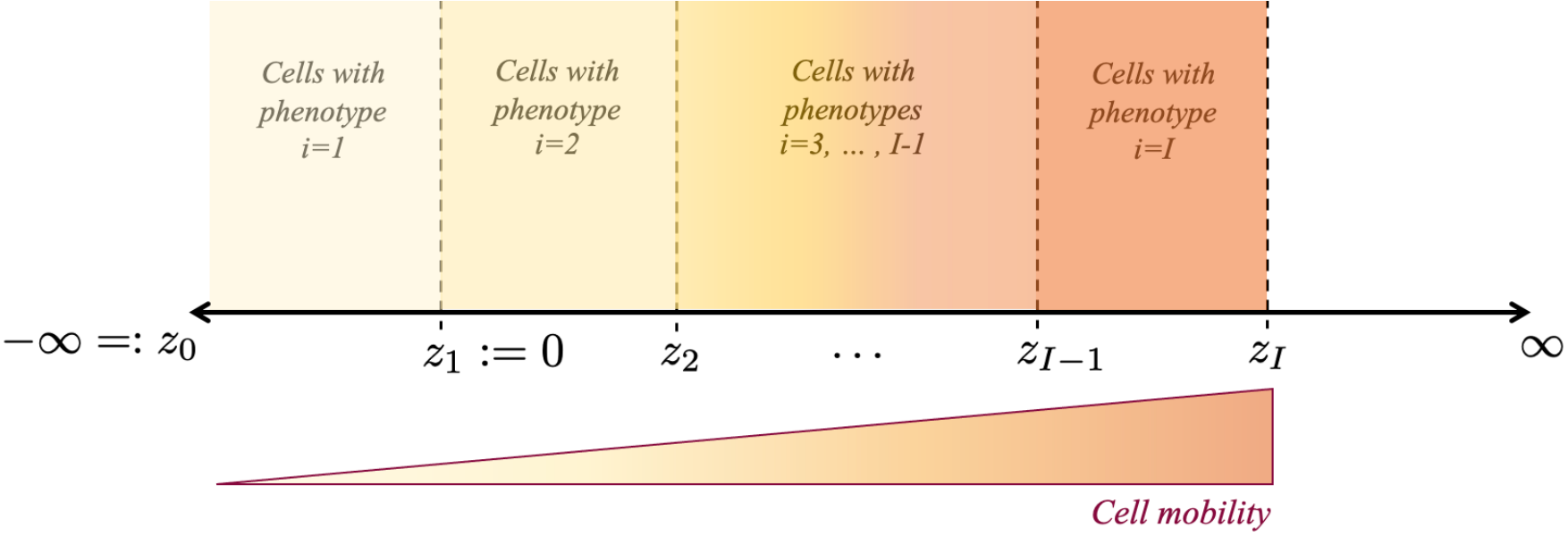}
\caption{{\bf Schematic overview of spatial segregation across travelling waves.} Schematic of how, under assumptions~\eqref{eq:G_conditions}-\eqref{eq:alpha_conditions} and~\eqref{ass:modpar3}, cells with phenotypes labelled by different values of the index $i$ are spatially segregated across invading fronts, which are represented by travelling wave solutions of the continuum model~\eqref{eq:PDEmodel}, i.e. solutions of the system of differential equations~\eqref{eq:PDEmodelTW}, subject to conditions \eqref{ass:SuppTW1}-\eqref{ass:SuppTW2}}
\label{fig:Fig2}
\end{figure} 
The properties of such travelling wave solutions are established by Theorem~\ref{th:theo1}.
\\\\
\begin{theorem}
\label{th:theo1}
Let assumptions~\eqref{eq:G_conditions}-\eqref{eq:alpha_conditions} and \eqref{ass:modpar3} hold. For any $M_2, \ldots, M_I \in \mathbb{R}^*_+$ there exist $z_2, \ldots, z_{I} \in \mathbb{R}^*_+$ and $c \in \mathbb{R}^*_+$ such that the system of differential equations~\eqref{eq:PDEmodelTW} subject to conditions~\eqref{ass:SuppTW1}-\eqref{ass:asycon} admits a solution wherein each component $n_i(z)$ is positive, continuous, and decreasing on $(z_{i-1}, z_i)$ for $i=1, \ldots, I$, and the components $n_{2}(z), \ldots, n_{I}(z)$ satisfy conditions~\eqref{eq:masscon} as well. Moreover, the function $p(z)$ defined via the constitutive relation~\eqref{eq:PDEmodelTW_2} is positive, continuous, and decreasing on $(-\infty,z_I)$, with
\beq
\label{eq:pat0}
p(0) =  \sqrt{2 \, c \, \sum_{j=2}^I \dfrac{\omega_j}{\mu_j} \, M_j},
\eeq
and it has a kink at the points  $0, z_2, \ldots, z_{I-2}, z_{I-1}$ with 
\beq
\label{eq:kinkPns}
-\mu_{2} \, p'(0^+) = -\mu_{1} \, p'(0^-) = c,  \;\; \text{ and } \;\; -\mu_{i+1} \, p'(z_i^+) = -\mu_{i} \, p'(z_i^-) = c, \; i=2,\ldots, I-1.
\eeq
\end{theorem}

\begin{proof}
We prove Theorem~\ref{th:theo1} in 8 steps. In summary, building on the shooting method that we employed in~\cite{lorenzi2017interfaces}, first we prove, for $c \in \mathbb{R}^*_+$ given, that: (i) The cell pressure $p(z)$ and the cell density $n_i(z)$ are positive, continuous, and monotonically decreasing on the intervals $(z_{i-1},z_{i})$ for all $i=1,\ldots,I$. Then $p(z)$ is not only monotonically decreasing on each interval but as a whole in $\mathbb{R}$, with possibly jumps at the points $z_i$ for $i=1,\ldots,I-1$. This along with the non-negativity of $p(z)$ ensures that $0 \leq p(z) \leq \overline{p}$ for all $z \in \mathbb{R}$  ({\it Step 1}). (ii) The cell pressure is continuous also in $z_i$ for all $i=1,\ldots,I$ ({\it Step 2}). (iii) The derivative of the cell pressure satisfies conditions~\eqref{eq:kinkPns} ({\it Steps 3-4}). Then, still for $c \in \mathbb{R}^*_+$ given, imposing conditions~\eqref{eq:masscon}, we find the values attained by the cell pressure $p(z)$ at the points $z_2, \ldots, z_{I-1}$, and we prove that condition~\eqref{eq:pat0} holds ({\it Steps 5-6}). Next, still for $c \in \mathbb{R}^*_+$ given, we identify the points $z_2, \ldots, z_{I}$ ({\it Step 7}). Finally, we prove that there exists a unique pair $(c,p)$ that satisfies the travelling wave problem ({\it Step 8}).
\\\\
{\it Preliminary observations.}  \textcolor{black}{ Throughout the proof we will be exploiting the fact that the system of differential equations~\eqref{eq:PDEmodelTW} subject to conditions~\eqref{ass:SuppTW1}-\eqref{ass:SuppTW2} can be rewritten as 
\begin{subnumcases}{\label{eq:PDEmodelTWred}}
  -c \, n_1' - \mu_1 \, \left(n_1 \, p' \right)' = G_1(p) \, n_1, \quad z \in (z_0, z_1) \equiv (-\infty,0),& \label{eq:PDEmodelTWred_1}
   \\
      \nonumber\\
   -c \, n_i' - \mu_i \, \left(n_i \, p' \right)' =  0, \quad z \in (z_{i-1}, z_i) \quad i=2, \ldots, I,  \label{eq:PDEmodelTWred_1a}\\
   \nonumber\\
   p(z) :=  \omega_i \, n_i(z), \quad z \in (z_{i-1}, z_i), \quad i=1,\ldots,I. &  \label{eq:PDEmodelTWred_2}
\end{subnumcases}
}
 \textcolor{black}{Moreover, multiplying both sides of the differential equation~\eqref{eq:PDEmodelTWred_1} by $\omega_{1}$ and both sides of the differential equation~\eqref{eq:PDEmodelTWred_1a} by $\omega_{i}$, using the relation between $p$ and $n_1$ and $p$ and $n_i$ given by~\eqref{eq:PDEmodelTWred_2} along with the fact that $G_1(p):= \alpha_1 \, G(p)$ where $\alpha_1 > 0$ (cf. assumptions~\eqref{eq:G_conditions}-\eqref{eq:alpha_conditions}), we obtain the following set of differential equations for $p$ }
 \textcolor{black}{ 
\begin{subnumcases}{\label{e.P-inf0FB12popi}}
  -c \, p' - \mu_1 \, \left(p \, p' \right)' = \alpha_1 \, G(p) \, p, \quad z \in (z_0, z_1) \equiv (-\infty,0),& \label{e.P-inf0FB12popi_1}
   \\
      \nonumber\\
   -c \, p' - \mu_i \, \left(p \, p' \right)' =  0, \quad z \in (z_{i-1}, z_i) \quad i=2, \ldots, I.  \label{e.P-inf0FB12popi_2}
\end{subnumcases}
}
 \textcolor{black}{Furthermore, we also notice that, introducing the notation
\beq
\label{defmu}
\mu(z) := \sum_{i=1}^I \mu_i \, \mathbbm{1}_{(z_{i-1},z_{i})}(z), 
\eeq
where $\mathbbm{1}_{(\cdot)}(z)$ is the indicator function of the set $(\cdot)$, and using the fact that, under conditions~\eqref{ass:SuppTW1}-\eqref{ass:SuppTW2}, both $p(z) = 0$ and $p'(z) = 0$ for all $z >z_I$, the set of differential equations~\eqref{e.P-inf0FB12popi_2} can be rewritten in a more compact form as
\beq
\label{eqpcompzminz1}
-c \, p'(z) - \left(\mu(z) \, p(z) \, p'(z) \right)' = 0, \quad z \in (z_1,\infty) \equiv (0,\infty).
\eeq
}
\\
\noindent {\it Step 1.} \textcolor{black}{For $c \in \mathbb{R}^*_+$ given, under assumptions~\eqref{eq:G_conditions} on the function $G(p)$, the differential equation~\eqref{e.P-inf0FB12popi_1} subject to the condition~\eqref{ass:asycon} admits solutions which are positive, continuous and, by the maximum principle, monotonically decreasing on $(z_0,z_1) \equiv (-\infty,0)$. Moreover, still for $c \in \mathbb{R}^*_+$ given, integrating the differential equation~\eqref{eqpcompzminz1} between $z \in (z_{i-1},z_i)$ with $i=2, \ldots, I$ and $\infty$, and using the fact that both $p(z) \to 0$ and $p'(z) \to 0$ as $z \to \infty$, yields}
\beq
\label{e.Pon0rFB22a}
\textcolor{black}{p'(z) = -\dfrac{c}{\mu_i}  \quad z \in (z_{i-1},z_i), \quad  i=2, \ldots, I,}
\eeq
\textcolor{black}{from which we see that the differential equation~\eqref{eqpcompzminz1} admits solutions which are positive, continuous, and monotonically decreasing on $(z_{i-1},z_{i})$ for all $i=2,\ldots,I$. In particular, note that integrating the differential equation~\eqref{e.Pon0rFB22a} with $i=I$ between a generic point $z \in (z_{I-1},z_I)$ and $z_I$ and imposing the condition $p(z_I)=0$ (cf. the conditions~\eqref{ass:SuppTW1}-\eqref{ass:SuppTW2}) we obtain
\beq
\label{e:PeqFB}
p(z) = \dfrac{c}{\mu_I} \, (z_I - z), \quad z \in (z_{I-1},z_I).
\eeq
}

\textcolor{black}{In summary, $p(z)$ is positive, continuous, and monotonically decreasing on $(z_{i-1}, z_i)$ for all $i=1, \ldots, I$. This along with the relations~\eqref{eq:PDEmodelTWred_2} allow us to conclude that also $n_i(z)$ is positive, continuous, and monotonically decreasing on $(z_{i-1}, z_i)$ for all $i=1, \ldots, I$.}
\\\

\noindent {\it Step 2.} For $c \in \mathbb{R}^*_+$ given, under assumptions~\eqref{eq:G_conditions}-\eqref{eq:alpha_conditions},  \textcolor{black}{introducing the additional notation
\beq
\label{defalpha}
\alpha(z) := \sum_{i=1}^I \alpha_i \, \mathbbm{1}_{(z_{i-1},z_{i})}(z),
\eeq
and using the definition~\eqref{defmu} of $\mu(z)$ along with the fact that, under conditions~\eqref{ass:SuppTW1}-\eqref{ass:SuppTW2}, both $p(z) = 0$ and $p'(z) = 0$ for all $z >z_I$, we further} rewrite the set of differential equations~\eqref{e.P-inf0FB12popi} in a more compact form as
\beq
\label{eqpcomp}
-c \, p'(z) - \left(\mu(z) \, p(z) \, p'(z) \right)' = \alpha(z) \, G(p) \, p(z), \quad z \in \mathbb{R}.
\eeq
Multiplying by $p$ both sides of the differential equation~\eqref{eqpcomp} and rearranging terms we find 
\beq
\label{eqpcompbyp}
\mu \, p \left(p' \right)^2 = \alpha \, G(p) \, p^2 + \left(p \, \mu \, p \,  p' \right)' + c \, p \, p', \quad z \in \mathbb{R}.
\eeq
Since the function $p(z)$ is continuous on $(z_{j-1},z_j)$ for all $j=1, \ldots, I-1$, there exists $z_{j}^* \in (z_{j-1},z_{j})$ such that $p'(z_{j}^*) > -\infty$ for any $j=1, \ldots, I-1$. Hence, integrating both sides of~\eqref{eqpcompbyp} between $z^*_j$ and $z_{j+1}$ and estimating the right-hand side from above, by using the fact that the functions $p$, $\mu$, $\alpha$\textcolor{black}{, and $G(p)$} are non-negative and bounded on $(z_{j-1}, z_{j+1})$ while the function $p'$ is non-positive on $(z_{j-1}, z_{j+1})$, we obtain
$$
\int_{z^*_j}^{z_{j+1}}  \mu \, p \, \left(p' \right)^2 {\rm d}z  < \infty, \quad\quad j=1, \ldots, I-1.
$$
The above estimates ensure that $p' \in L^2_{loc}\left((z^*_j, z_{j+1})\right)$ for $j=1, \ldots, I-1$. This along with the fact that $p \in L^{\infty}(\mathbb{R})$ allow us to conclude that $p$ is also continuous in each $z_i$ for $i=1, \ldots, I-1$, i.e.
\beq
\label{eq:contP}
p(z_{i}^+) = p(z_{i}^-) = p(z_i), \quad i=1,\ldots,I-1.
\eeq
\\
\noindent {\it Step 3.} For $c \in \mathbb{R}^*_+$ given, integrating the differential equation~\eqref{eqpcomp} between a generic point $z \in (z_{j-1},z_{j+1})$ and $z_{j+1}$ for $j=1,\ldots,I-1$, and using the fact that $p(z)$ is continuous and $\mu(z_{j+1}^-)=\mu_{j+1}$ (cf. the definition~\eqref{defmu} of $\mu(z)$), yields
\beq
\label{neqrev0}
c \, p(z) + \mu(z) \, p(z) \, p'(z) =  \int_{z}^{z_{j+1}} \alpha \, G(p) \, p \, {\rm d}\zeta + \, c \, p(z_{j+1}) + \mu_{j+1} \, p(z_{j+1}) \, p'(z_{j+1}^-).
\eeq

Letting $z \to z_j^-$ in \eqref{neqrev0} and using the fact that $p(z)$ is continuous and $\mu(z_{j}^-)=\mu_{j}$ (cf. the definition~\eqref{defmu} of $\mu(z)$) gives
\beq
\label{neqrev1}
c \, p(z_j) + \mu_j \, p(z_j) \, p'(z_j^-) =  \int_{z_j}^{z_{j+1}} \alpha \, G(p) \, p \, {\rm d}z + \, c \, p(z_{j+1}) + \mu_{j+1} \, p(z_{j+1}) \, p'(z_{j+1}^-).
\eeq
Similarly, letting $z \to z_j^+$ in \eqref{neqrev0} and using the fact that $p(z)$ is continuous and $\mu(z_{j}^+)=\mu_{j+1}$ (cf. the definition~\eqref{defmu} of $\mu(z)$) gives
\beq
\label{neqrev2}
c \, p(z_j) + \mu_{j+1} \, p(z_j) \, p'(z_j^+) =  \int_{z_j}^{z_{j+1}} \alpha \, G(p) \, p \, {\rm d}z + \, c \, p(z_{j+1}) + \mu_{j+1} \, p(z_{j+1}) \, p'(z_{j+1}^-).
\eeq
Combining~\eqref{neqrev1} and~\eqref{neqrev2} we obtain
\beq
\label{eq:jumpnew2}
\mu_{i+1} \, p'(z_i^+) = \mu_{i} \, p'(z_i^-), \quad i=1,\ldots, I-1.
\eeq
\\

\noindent {\it Step 4.} For $c \in \mathbb{R}^*_+$ given, combining~\eqref{e.Pon0rFB22a} with $i=I$ (i.e. the fact that $c = - \mu_I \, p'(z_{I-1}^+)$) and the condition~\eqref{eq:jumpnew2} with $i=I-1$ (i.e. the fact that $\mu_{I-1} \, p'(z_{I-1}^-)=\mu_I \, p'(z_{I-1}^+)$) we obtain 
\beq
\label{eq:jumpnew2Fm}
- \mu_{I-1} \, p'(z_{I-1}^-) = - \mu_I \, p'(z_{I-1}^+) = c.
\eeq
Moreover, integrating the differential equation~\eqref{eqpcompzminz1} between $z_{i}$ and $z_{i+1}$ for $i=1,\ldots,I-2$ and using the fact that $p(z)$ is continuous and $\mu(z_{i+1}^-)=\mu(z_{i}^+)=\mu_{i+1}$  (cf. the definition~\eqref{defmu} of $\mu(z)$) yields 
$$
- c \left(p(z_{i+1}) - p(z_{i})\right) - \mu_{i+1} \left(p(z_{i+1}) \, p'(z_{i+1}^-) - p(z_{i}) \, p'(z_{i}^+) \right) = 0, \quad i=1,\ldots,I-2,
$$
from which, rearranging terms, we find
\beq
\label{eq:jumpnew2Fma}
p(z_{i}) \, \left(c + \mu_{i+1} p'(z_{i}^+)\right) -p(z_{i+1}) \, \left(c + \mu_{i+1} \, p'(z_{i+1}^-) \right) = 0, \quad i=1,\ldots,I-2.
\eeq
Choosing $i=I-2$ in~\eqref{eq:jumpnew2Fma} gives
$$
p(z_{I-2}) \, \left(c + \mu_{I-1} p'(z_{I-2}^+)\right) -p(z_{I-1}) \, \left(c + \mu_{I-1} \, p'(z_{I-1}^-) \right) = 0
$$
and then substituting~\eqref{eq:jumpnew2Fm} into the above equation yields
$$
p(z_{I-2}) \, \left(c + \mu_{I-1} \, p'(z_{I-2}^+)\right) = 0 \quad \Longrightarrow \quad - \mu_{I-1} \, p'(z_{I-2}^+) = c,
$$
from which, using condition~\eqref{eq:jumpnew2} with $i=I-2$ (i.e. the fact that $\mu_{I-1} \, p'(z_{I-2}^+) = \mu_{I-2} \, p'(z_{I-2}^-)$), we obtain
\beq
\label{eq:jumpnew2Fm1}
-\mu_{I-2} \, p'(z_{I-2}^-) = -\mu_{I-1} \, p'(z_{I-2}^+) = c.
\eeq

Proceeding in a similar way for $i=I-3$ and so on and so forth for $i=I-4, \ldots, 1$ we also find 
\beq
\label{eq:jumpnew2Fm2}
- \mu_{i} \, p'(z_{i}^-) = - \mu_{i+1} \, p'(z_{i}^+) = c, \quad i=1,\ldots,I-3.
\eeq
Taken together, the results given by~\eqref{eq:jumpnew2Fm}, \eqref{eq:jumpnew2Fm1}, and~\eqref{eq:jumpnew2Fm2} allow us to conclude that 
\beq
\label{eq:jumpnew2Fm1aa}
-\mu_{i+1} \, p'(z_i^+) = -\mu_{i} \, p'(z_i^-) = c, \quad i=1,\ldots, I-1,
\eeq
and thus conditions~\eqref{eq:kinkPns} hold. 
\\\\
{\it Step 5.} For $c \in \mathbb{R}^*_+$ given, the expression~\eqref{e:PeqFB} for $p$ on $(z_{I-1},z_I)$ along with the constitutive relation~\eqref{eq:PDEmodelTWred_2} with $i=I$ yield
\beq
\label{e.Pon0rFB12fmnew}
n_{I}(z) = \dfrac{c}{\omega_I \, \mu_I} (z_I - z), \quad z \in (z_{I-1},z_I).
\eeq
Substituting~\eqref{e.Pon0rFB12fmnew} into~\eqref{eq:masscon} with $i=I$, \textcolor{black}{and taking the positive root so as to ensure that the condition $z_{I}-z_{I-1}>0$ (i.e. $z_{I} > z_{I-1}$) holds}, gives
\beq
\label{e:zI}
\dfrac{c}{\omega_I  \, \mu_I} \int_{z_{I-1}}^{z_I} (z_I - z) \, {\rm d}z = M_I \quad \Longrightarrow \quad z_I = z_{I-1} + \sqrt{\dfrac{2 \, \omega_I \, \mu_I}{c} \, M_I}.
\eeq
Finally, substituting  the expression~\eqref{e:zI} for $z_I$ into~\eqref{e:PeqFB} and evaluating the resulting expression for $p(z)$ in $z=z_{I-1}$ yields
$$
p(z_{I-1}^+) = \sqrt{\dfrac{2 \, c \, \omega_I}{\mu_I} \, M_I},
$$
from which, exploiting the fact that $p(z)$ is continuous, we obtain
\beq
\label{e.Pon0rFB12fm}
p(z_{I-1}) = \sqrt{\dfrac{2 \, c \, \omega_I}{\mu_I} \, M_I}.
\eeq 
\\
\noindent {\it Step 6.} For $c \in \mathbb{R}^*_+$ given, integrating the differential equation~\eqref{e.Pon0rFB22a} between a generic point $z \in (z_{i-1},z_i)$ and $z_i$ for $i=2,\ldots,I-1$, we find 
\beq
\label{eqnew1}
p(z) = p(z_i) + \dfrac{c}{\mu_i} (z_i - z), \quad z \in (z_{i-1}, z_i), \quad i=2,\ldots,I-1.
\eeq
Choosing $i=I-1$ in~\eqref{eqnew1} gives 
\beq
\label{eqnew2}
p(z) = p(z_{I-1}) + \dfrac{c}{\mu_{I-1}} (z_{I-1} - z), \quad z \in (z_{I-2}, z_{I-1}). 
\eeq
Integrating~\eqref{eqnew2} between $z_{I-2}$ and $z_{I-1}$, using the fact that $p(z) = \omega_{I-1} n_{I-1}(z)$ for $z \in (z_{I-2}, z_{I-1})$ (cf. the constitutive relation~\eqref{eq:PDEmodelTWred_2} with $i=I-1$), and imposing the condition~\eqref{eq:masscon} with $i=I-1$ gives
$$
\int_{z_{I-2}}^{z_{I-1}} \left(p(z_{I-1}) + \dfrac{c}{\mu_{I-1}} (z_{I-1} - z) \right) \, {\rm d}z = \omega_{I-1} \, M_{I-1},
$$
from which, computing the integral, solving the resulting quadratic equation for $z_{I-1}-z_{I-2}$, and taking the positive root so as to ensure that the condition $z_{I-1}-z_{I-2}>0$ (i.e. $z_{I-1} > z_{I-2}$) holds, we find   
\beq
\label{eqnew3}
z_{I-1} - z_{I-2} = \sqrt{\left(\dfrac{\mu_{I-1}}{c}\right)^2 (p(z_{I-1}))^2 + 2 \, \dfrac{\mu_{I-1}}{c} \, \omega_{I-1} \, M_{I-1}} - \dfrac{\mu_{I-1}}{c} \, p(z_{I-1}).
\eeq
Moreover, evaluating~\eqref{eqnew2} in $z_{I-2}$ and substituting~\eqref{eqnew3} into the resulting equation yields
\begin{eqnarray*}
p(z_{I-2}) &=& \dfrac{c}{\mu_{I-1}} \, \sqrt{\left(\dfrac{\mu_{I-1}}{c}\right)^2 (p(z_{I-1}))^2 + 2 \, \dfrac{\mu_{I-1}}{c} \, \omega_{I-1} \, M_{I-1}} 
\\
&=& \sqrt{(p(z_{I-1}))^2 + \dfrac{2 \, c \, \omega_{I-1}}{\mu_{I-1}} \, \, M_{I-1}}.
\end{eqnarray*}
Finally, substituting the expression~\eqref{e.Pon0rFB12fm} for $p(z_{I-1})$ into the above equation we obtain
\beq
\label{eqnew4}
p(z_{I-2}) = \sqrt{\dfrac{2 \, c \, \omega_I}{\mu_I} \, M_I + \dfrac{2 \, c \, \omega_{I-1}}{\mu_{I-1}} \, \, M_{I-1}}.
\eeq

Proceeding in a similar way for $i=I-2$ and so on and so forth for $i=I-3, \ldots, 2$ we also find 
\beq
\label{eqnew5b}
p(z_i) = \sqrt{2 \, c \, \sum_{j=i+1}^I \dfrac{\omega_j}{\mu_j} \, M_j}, \quad i=2,\ldots,I-2.
\eeq
Taken together, the results given by~\eqref{e.Pon0rFB12fm}, \eqref{eqnew4}, and~\eqref{eqnew5b} allow us to conclude that 
\beq
\label{eqnew5}
p(z_i) = \sqrt{2 \, c \, \sum_{j=i+1}^I \dfrac{\omega_j}{\mu_j} \, M_j}, \quad i=1,\ldots,I-1.
\eeq
Choosing $i=1$ in~\eqref{eqnew5} and recalling that $z_1:=0$ (cf. conditions~\eqref{ass:SuppTW2}), we obtain 
\beq
\label{eqnew6}
p(0) = \sqrt{2 \, c \, \sum_{j=2}^I \dfrac{\omega_j}{\mu_j} \, M_j},
\eeq
which implies that condition~\eqref{eq:pat0} holds. 
\\\\
\noindent {\it Step 7.} Evaluating~\eqref{eqnew1} in $z_{i-1}$ and solving for $z_{i}$ yields
\beq
\label{eqnew7}
z_i = z_{i-1} + \dfrac{\mu_i}{c} \left(p(z_{i-1}) - p(z_i)\right), \quad i=2, \ldots, I-1.
\eeq
Choosing $i=2$ in~\eqref{eqnew7} and recalling that $z_1:=0$ (cf. conditions~\eqref{ass:SuppTW2}) gives
$$
z_2 = \dfrac{\mu_2}{c} \left(p(0) - p(z_2)\right)
$$
and then, substituting into the above equation the expression~\eqref{eqnew6} for $p(0)$ and the expression for $p(z_2)$ obtained by choosing $i=2$ in~\eqref{eqnew5}, we find
\beq
\label{eq:revrev1}
\textcolor{black}{z_2 =  \dfrac{\mu_2}{c} \left(\sqrt{2 \, c \, \sum_{j=2}^I \dfrac{\omega_j}{\mu_j} \, M_j} - \sqrt{2 \, c \, \sum_{j=3}^I \dfrac{\omega_j}{\mu_j} \, M_j}\ \right).} 
\eeq
\textcolor{black}{Moreover, proceeding in a similar way for $i=3$ and so on and so forth for $i=4, \ldots, I-1$ we also obtain}
\beq
\label{eq:revrev2}
\textcolor{black}{z_i = z_{i-1} + \dfrac{\mu_i}{c} \left(\sqrt{2 \, c \, \sum_{j=i}^I \dfrac{\omega_j}{\mu_j} \, M_j} - \sqrt{2 \, c \, \sum_{j=i+1}^I \dfrac{\omega_j}{\mu_j} \, M_j}\ \right), \quad i=3, \ldots, I-1.}
\eeq
\textcolor{black}{Finally, from~\eqref{e:zI} we find }
\beq
\label{eq:revrev3}
\textcolor{black}{z_I = z_{I-1} + \sqrt{\dfrac{2 \, \omega_I \, \mu_I}{c} \, M_I},}
\eeq
\textcolor{black}{with $z_{I-1}$ obtained from~\eqref{eq:revrev2} by choosing $i=I-1$.}
\\\\
\noindent {\it Step 8.} \textcolor{black}{Complementing the differential equation~\eqref{e.P-inf0FB12popi_1} with the condition~\eqref{ass:asycon} and the condition~\eqref{eq:jumpnew2Fm1aa} for $i=1$}, we obtain the following problem
\begin{subnumcases}{\label{eq:TW2}}
-c \, p'(z) - \mu_1 \, \left(p(z) \, p'(z) \right)' = \alpha_1 \, G(p) \, p(z), \quad z \in (z_0, z_1) \equiv (-\infty,0), & \label{eq:TW2_1}
   \\
   \nonumber\\
p(-\infty) = \overline{p}, \quad p'(0^-) = -\dfrac{c}{\mu_{1}}, \label{eq:TW2_2}
\end{subnumcases}
from which, proceeding as similarly done in~\cite{lorenzi2017interfaces}, it is possible to prove \textcolor{black}{(see Appendix~\ref{app:appendixC})} that $c \mapsto p(0^-)$ is monotonically decreasing. On the other hand, the expression~\eqref{eqnew6} for $p(0)$ implies that $c \mapsto p(0^+)$  is monotonically increasing. These facts along with the condition $p(0^-)=p(0^+)$, which follows from the continuity of $p(z)$, allow us to conclude that there exists a unique pair $(c,p)$ that satisfies the travelling wave problem. 
\end{proof}

\begin{remark}
\label{rem:remth1}
The fact that, as established by Theorem~\ref{th:theo1}, the cell pressure $p(z)$ defined via the constitutive relation~\eqref{eq:PDEmodelTW_2} is continuous throughout the support of the travelling wave gives the following interface conditions for the cell densities
\beq
\label{eq:sgndiffn}
n_{i+1}(z_{i}^+) = \dfrac{\omega_{i}}{\omega_{i+1}} n_{i}(z_{i}^-), \quad i=1, \ldots, I-1.
\eeq
These conditions imply that if $\omega_{i}<\omega_{i+1}$ then $n_{i+1}(z_{i}^+)<n_{i}(z_{i}^-)$, whereas if $\omega_{i} \geq \omega_{i+1}$ then $n_{i+1}(z_{i}^+) \geq n_{i}(z_{i}^-)$.
\end{remark}
\bigskip
\begin{remark}
\label{rem:remth12}
The conditions~\eqref{eq:kinkPns} imply that
$$
-p'(z_i^+) = - \dfrac{\mu_i}{\mu_{i+1}} \, p'(z_i^-), \quad i=1, \ldots, I-1
$$
and, therefore, since $p'(z_i^+)<0$ and $p'(z_i^-)<0$ for all $i=1, \ldots, I-1$, we have
$$
|p'(z_i^+)| = \dfrac{\mu_i}{\mu_{i+1}} \, |p'(z_i^-)|, \quad i=1, \ldots, I-1.
$$
Under assumptions~\eqref{ass:modpar3}, which imply that $ \dfrac{\mu_i}{\mu_{i+1}}<1$ for all $i=1, \ldots, I-1$, the above relations allow us to conclude that
$$
|p'(z_i^+)| < |p'(z_i^-)|, \quad i=1, \ldots, I-1.
$$
\end{remark}

\section{Numerical simulations}\label{sec:section5}
In this section, we present results of numerical simulations of the individual-based model introduced in Section~\ref{sec:section2} and the corresponding continuum model defined by the PDE system~\eqref{eq:PDEmodel}, and compare them to the results of travelling wave analysis obtained in Section~\ref{sec:section4}. In particular, we investigate the cases where there are either three or four different cellular phenotypes (i.e. $I=3$ or $I=4$). 

\subsection{Set-up of numerical simulations}\label{sec:section5p1}
We carry out numerical simulations over the spatial domain $[0,L]$, with $L=150$. In order to ensure that assumptions~\eqref{eq:G_conditions} are satisfied, we use the following definition of the function $G(p)$, both when $I=3$ and when $I=4$,
\begin{equation}
G(p):=\arctan\left(\dfrac{1}{10}\left(1-\frac{p}{\overline{p}}\right)\right).
\label{eq:DefGrowth}
\end{equation}
Furthermore, we choose values of the parameters $\alpha_i$ and $\mu_i$ satisfying assumptions~\eqref{eq:alpha_conditions} and~\eqref{ass:modpar3}. Specifically, in the case where $I=3$ we use the following baseline parameter values 
\begin{eqnarray}
\label{eq:parametersI=3}
&&\alpha_1=10,\quad \alpha_2=\alpha_3=0,\nonumber\\
&& \mu_1=10^{-4},\quad \mu_2=2\times10^{-4},\quad \mu_3=3\times10^{-4},\\
&& \omega_1=1,\quad \omega_2=2,\quad \omega_3=3,\nonumber
\end{eqnarray}
and in the case where $I=4$ we complement the parameter choice given by~\eqref{eq:parametersI=3} with 
\begin{eqnarray}
\label{eq:parametersI=4}
\alpha_4=0, \quad \mu_4=4\times10^{-4}, \quad \omega_4=4.
\end{eqnarray}
In both cases, we explore also deviations of the weights $\omega_i$ from these values in the numerical simulations.  

Moreover, in line with the travelling wave analysis carried out in Section~\ref{sec:section4}, \textcolor{black}{we investigate propagation of segregation properties by letting cells} with different phenotypes be spatially segregated at $t=0$. Specifically, both for the individual-based model and for the continuum model, when $I=3$ we define the initial cell densities as
\begin{eqnarray}
\label{eq:ICI=3}
n_1(0,x) &=& A_1 \exp\left[-B \, x^2\right] \mathbbm{1}_{[0,10)}(x),\nonumber \\
n_2(0,x) &=& A_2 \exp\left[-B \, (x-10)^2\right] \mathbbm{1}_{[10,20)}(x),\nonumber \\
n_3(0,x) &=& A_3 \exp\left[-B \, (x-20)^2\right] \mathbbm{1}_{[20,L)}(x),
\end{eqnarray}
while when $I=4$ we use the following definitions for the initial cell densities
\begin{eqnarray}
\label{eq:ICI=4}
n_1(0,x) &=& A_1 \exp\left[-B \, x^2\right] \mathbbm{1}_{[0,10)}(x),\nonumber \\
n_2(0,x) &=& A_2 \exp\left[-B \, (x-10)^2\right] \mathbbm{1}_{[10,20)}(x),\nonumber \\
n_3(0,x) &=& A_3 \exp\left[-B \, (x-20)^2\right] \mathbbm{1}_{[20,30)}(x),\nonumber \\
n_4(0,x) &=& A_4 \exp\left[-B \, (x-30)^2\right] \mathbbm{1}_{[30,L)}(x).
\end{eqnarray}
In~\eqref{eq:ICI=3} and~\eqref{eq:ICI=4}, the function $\mathbbm{1}_{(\cdot)}(x)$ is the indicator function of the set $(\cdot)$, $B=6\times10^{-2}$, and the parameters $A_i$ are positive real numbers. We choose the homeostatic pressure to be
$$
\overline{p} = 4\times10^4
$$
and, given this value of $\overline{p}$ and the values selected for the weights $\omega_i$ in~\eqref{eq:PDEmodel_3} and~\eqref{eq:pressure_definition}, we then choose the values of the parameters $A_i$ in~\eqref{eq:ICI=3} and~\eqref{eq:ICI=4} such that the initial cell pressure satisfies $p(0,x)\leq \overline{p}$ for all $x \in [0,L]$ and is consistent between numerical simulations. 

Finally, to carry out numerical simulations of the individual-based model, we choose the space-step $\Delta_x=0.1$, the time-step $\tau=1\times10^{-4}$, and we define
\beq
\label{def:gammasim}
\gamma_i := \frac{2\tau \overline{p}}{\Delta_x^2} \mu_i
\eeq
so as to ensure that conditions~\eqref{ass:formderi} underlying the formal derivation of the continuum model are met. Note that, since we choose values of the parameters $\mu_i$  satisfying assumptions~\eqref{ass:modpar3}, defining the values of the parameter $\gamma_i$ according to~\eqref{def:gammasim} ensures that assumptions~\eqref{eq:gamma_conditions} are satisfied as well. 

\subsection{Computational implementation of the individual-based model and numerical scheme for the continuum model}\label{sec:section5p2}
All simulations are performed in {\sc{Matlab}}. For the individual-based model, at each time-step, every individual cell can undergo: (i) movement, according to the probabilities defined in Section~\ref{sec:section2p2}; (ii) division and death, according to the probabilities defined in Section~\ref{sec:section2p1}. For each of these processes a random number is drawn from the standard uniform distribution on the interval $(0,1)$ using the built-in {\sc{Matlab}} function {\sc{rand}}. If this random number is smaller than the probability of the event occurring then the process is successful. To impose zero-flux boundary conditions, any attempted move outside the spatial domain is aborted. To numerically solve the PDE system~\eqref{eq:PDEmodel} subject to zero-flux boundary conditions we use a finite volume scheme modified from our previous works~\cite{bubba2020discrete,lorenzi2023derivation}. The full details of this scheme are provided in Appendix~\ref{app:appendixB}.

\subsection{Main results of numerical simulations}\label{sec:section5p4}
Figures~\ref{fig:Fig3} and~\ref{fig:Fig4} display the results of numerical simulations of the individual-based model (bottom panels) and the corresponding continuum model defined by the PDE system~\eqref{eq:PDEmodel} (top panels) for $I=3$ and $I=4$, respectively, under the baseline parameter settings~\eqref{eq:parametersI=3} and~\eqref{eq:parametersI=3}-\eqref{eq:parametersI=4}, respectively. Moreover, Figures~\ref{fig:Fig5} and~\ref{fig:Fig6} display the results of numerical simulations of the  continuum model for $I=3$ and $I=4$, respectively, when the values of the parameters $\alpha_i$ and $\mu_i$ are set according to~\eqref{eq:parametersI=3} and \eqref{eq:parametersI=3}-\eqref{eq:parametersI=4}, respectively, while different combinations of the weights $\omega_i$ are considered. 

\paragraph{Agreement between the individual-based and the continuum models} The results summarised by the plots in Figures~\ref{fig:Fig3} and~\ref{fig:Fig4} demonstrate that overall there is excellent agreement between numerical simulations of the individual-based model and numerical solutions of the corresponding continuum model, as expected since conditions~\eqref{ass:formderi} are met. Moreover, although, for the sake of clarity, only the results of numerical simulations of the continuum model are displayed in Figures~\ref{fig:Fig5} and~\ref{fig:Fig6}, we verified that also for the parameter settings corresponding to these figures the results of numerical simulations of the individual-based model agree with those of the continuum model (results not shown). \textcolor{black}{We stress that possible quantitative discrepancies between the two models emerge in the proximity of the interfaces between regions occupied by cells with different phenotypes (see Figure~\ref{fig:Fig34err}). This is because in these areas there is a stronger interplay between demographic stochasticity and sharp transitions in cell densities, which causes a reduction in the quality of the continuum approximations that are employed in the formal derivation of the PDE model from the underlying individual-based model. We also remark that the quantitative agreement between the two models may deteriorate in scenarios where a limited number of cells is considered, since in these scenarios the impact of demographic stochasticity is amplified.}

\begin{figure}[H]
\includegraphics[width=\textwidth,trim={4cm 1cm 4cm 0cm},clip]{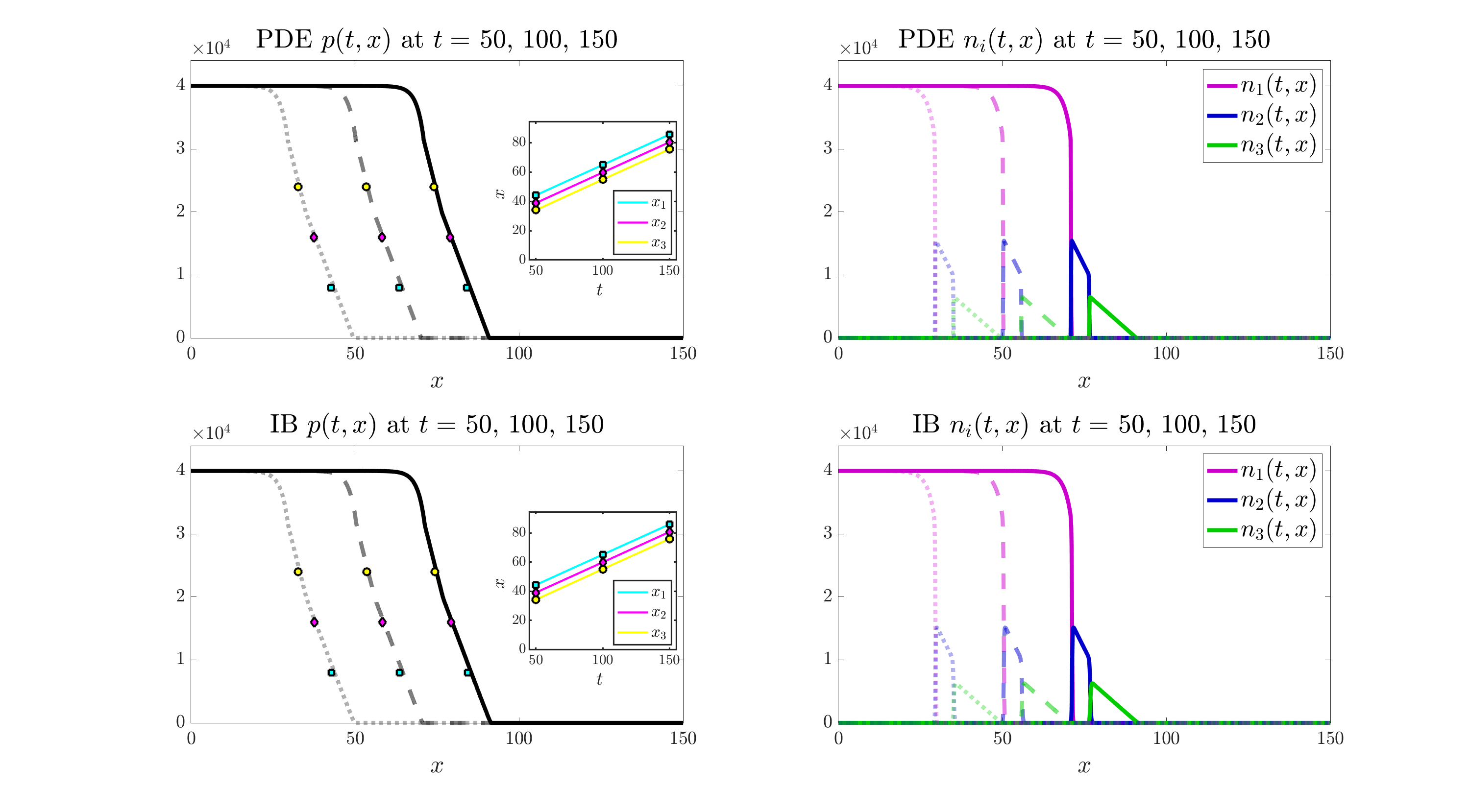}
\caption{{{\bf Main results under the baseline parameter setting for $I=3$.}} Comparing numerical solutions of the continuum model {(\bf{\textit{top panels}}}) with the averaged results of 10 simulations of the individual-based model {(\bf{\textit{bottom panels}}}), when $I=3$ and the values of the parameters $\alpha_i$, $\mu_i$, and $\omega_i$ are set according to~\eqref{eq:parametersI=3}. Plots display the cell pressure $p(t,x)$ {(\bf{\textit{left panels}}}) and the cell densities $n_i(t,x)$ {(\bf{\textit{right panels}}}) at three successive time instants -- i.e. $t=50$ \textcolor{black}{(dotted lines)}, $t=100$ \textcolor{black}{(dashed lines)}, and $t=150$ \textcolor{black}{(solid lines)}. The insets of the left panels display the plots of  $x_1(t)$ (cyan), $x_2(t)$ (magenta), and $x_3(t)$ (yellow) defined via~\eqref{eq:xplots}. The coloured markers in the plot of $p(t,x)$ highlight the values of $p(t,x_1(t))$ (cyan), $p(t,x_2(t))$ (magenta), and $p(t,x_3(t))$ (yellow) at $t=50$, $t=100$, and $t=150$. The numerically estimated wave speeds are {{$c_{\text{PDEn}}= 0.42$ and $c_{\text{IBn}}=0.42$, and the analytically predicted wave speed is $c_{\text{a}}=0.42$}}} 
\label{fig:Fig3}
\end{figure} 

\begin{figure}[H]
\includegraphics[width=\textwidth,trim={4cm 1cm 4cm 0cm},clip]{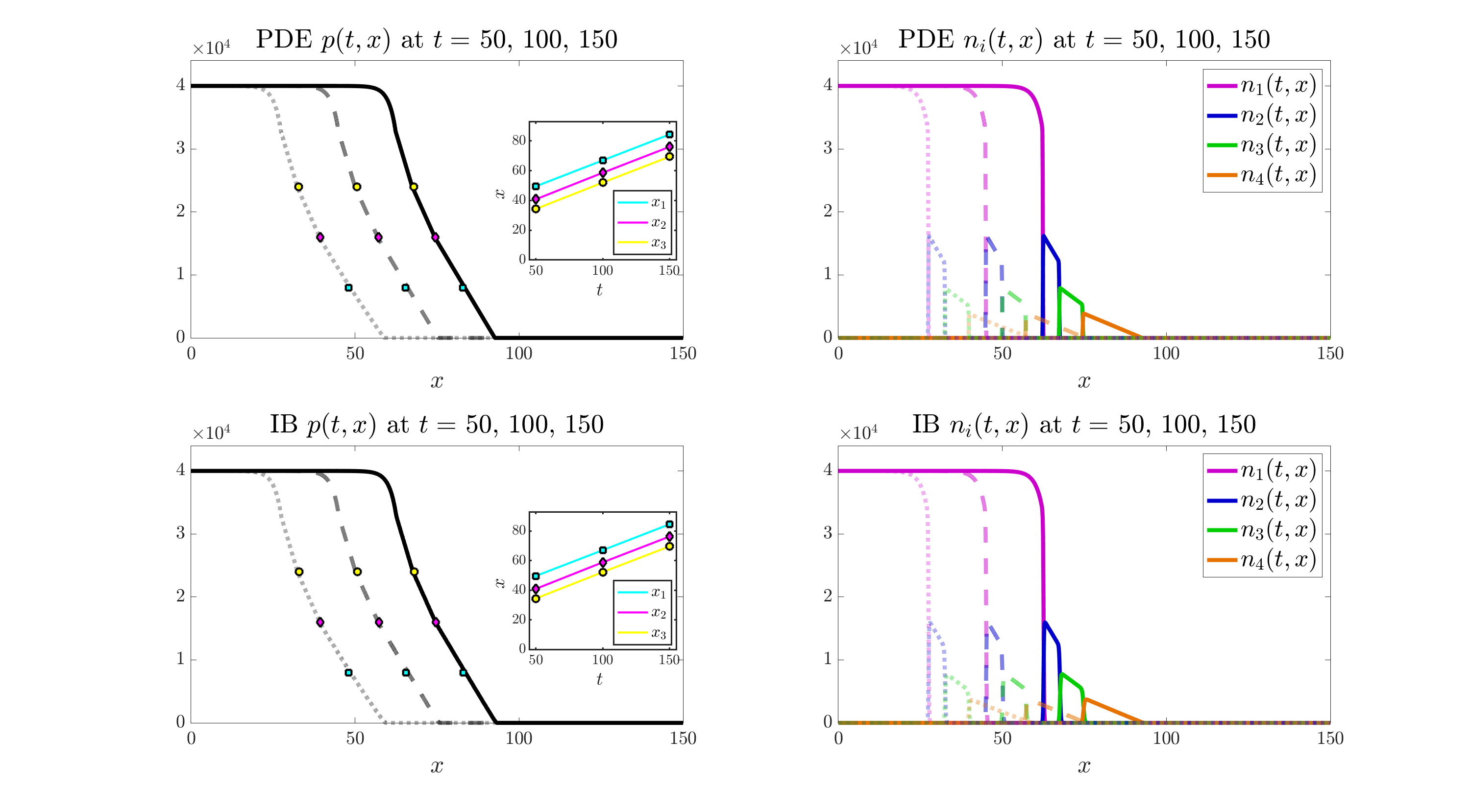}
\caption{{{\bf Main results under the baseline parameter setting for $I=4$.}} Comparing numerical solutions of the continuum model {(\bf{\textit{top panels}}}) with the averaged results of 10 simulations of the individual-based model {(\bf{\textit{bottom panels}}}), when $I=4$ and the values of the parameters $\alpha_i$, $\mu_i$, and $\omega_i$ are set according to~\eqref{eq:parametersI=3}-\eqref{eq:parametersI=4}. Plots display the cell pressure $p(t,x)$ {(\bf{\textit{left panels}}}) and the cell densities $n_i(t,x)$ {(\bf{\textit{right panels}}}) at three successive time instants -- i.e. $t=50$ \textcolor{black}{(dotted lines)}, $t=100$ \textcolor{black}{(dashed lines)}, and $t=150$ \textcolor{black}{(solid lines)}. The insets of the left panels display the plots of  $x_1(t)$ (cyan), $x_2(t)$ (magenta), and $x_3(t)$ (yellow) defined via~\eqref{eq:xplots}. The coloured markers in the plot of $p(t,x)$ highlight the values of $p(t,x_1(t))$ (cyan), $p(t,x_2(t))$ (magenta), and $p(t,x_3(t))$ (yellow) at $t=50$, $t=100$, and $t=150$. The numerically estimated wave speeds are {{$c_{\text{PDEn}}= 0.35$ and $c_{\text{IBn}}=0.35$, and the analytically predicted wave speed is $c_{\text{a}}=0.35$}}} 
\label{fig:Fig4}
\end{figure}

\begin{figure}[H]
\includegraphics[width=\textwidth]{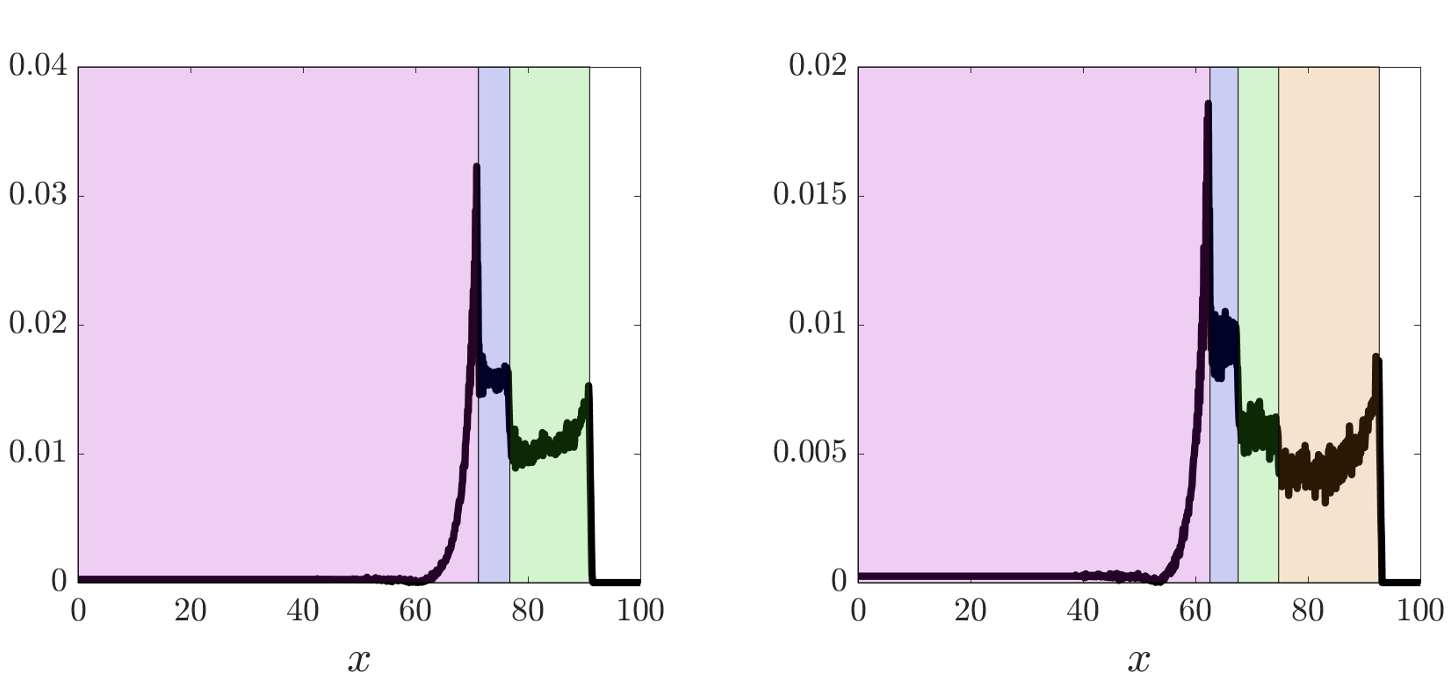}
\caption{{\textcolor{black}{{\bf Quantitative comparison between the individual-based and the continuum models.} Plot of the quantity $\dfrac{|p_{\text PDE}(t,x) - p_{\text IB}(t,x)|}{\overline{p}}$ at $t=150$, where $p_{\text PDE}$ is the cell pressure computed from numerical solutions of the continuum model displayed in Figure~\ref{fig:Fig3} (left panel) and Figure~\ref{fig:Fig4} (right panel), while  $p_{\text IB}$ is the cell pressure computed from the averaged results of 10 simulations of the individual-based model displayed in the same figures. The supports of the cell densities $n_i$ for $i=1,\ldots, I$ with $I=3$ (left panel) or $I=4$ (right panel) are highlighted in the same colours as those of the curves of the cell densities displayed in Figures~\ref{fig:Fig3} and~\ref{fig:Fig4}}}} 
\label{fig:Fig34err}
\end{figure} 

\paragraph{Propagation of travelling fronts} The numerical results in Figures~\ref{fig:Fig3}-\ref{fig:Fig4} and Figures~\ref{fig:Fig5}-\ref{fig:Fig6} show the propagation of travelling fronts wherein, as expected from Theorem~\ref{th:theo1}, the cell densities $n_i$ have disjoint supports, which means that cells with different phenotypes occupy distinct regions across the front (cf. right panels). In particular, as captured by assumptions~\eqref{eq:G_conditions}-\eqref{eq:alpha_conditions} and~\eqref{ass:modpar3}, cells with phenotypes labelled by larger values of the index $i$, which display a higher migratory ability, occupy regions closer to the edge of the front, while fast-dividing cells with the lowest mobility (i.e. cells with phenotype $i=1$) make up the bulk of the population in the rear of the front. Moreover, also in agreement with Theorem~\ref{th:theo1}, the cell pressure $p$ is continuous throughout the wave, whereas its first spatial derivative exhibits jump discontinuities at the interfaces between the regions occupied by cells with different phenotypes (cf. left panels). Specifically, after an initial transient during which the travelling front is formed, the numerical values of the first spatial derivative of the cell pressure $p$ are such that the interface conditions~\eqref{eq:kinkPns} are satisfied. 

In order to confirm the propagation of travelling waves, we also track the dynamics of the points $x_1(t)$, $x_2(t)$, and $x_3(t)$ such that 
\begin{equation}
p(t,x_1(t))=0.2 \, \overline{p}, \quad p(t,x_2(t))=0.4 \, \overline{p}, \quad p(t,x_3(t))=0.6 \, \overline{p}, \label{eq:xplots}
\end{equation} 
and verify that, after an initial transient during which the travelling wave is formed, the functions $x_1(t)$, $x_2(t)$, and $x_3(t)$ behave like straight lines with approximately the same constant slope (cf. insets of the left panels). We numerically estimate the wave speeds for the individual-based model and for the continuum model, denoted $c_{\text{IBn}}$ and $c_{\text{PDEn}}$, respectively, by measuring the slope of $x_1(t)$ after the transient. We then compare the numerically estimated wave speeds with the analytically predicted one, denoted $c_{\text{a}}$, which is obtained through~\eqref{eq:pat0}, i.e. substituting into the following formula
$$
c_{\text{a}} = \dfrac{\left(p(0)\right)^2}{\displaystyle{2 \sum_{j=2}^I \dfrac{\omega_j}{\mu_j} \, M_j}}
$$
the values of $M_2, \ldots, M_I$ and the value of $p(0)$ estimated from numerical solutions of the continuum model. In particular, we approximate $M_j$ with the numerical value of the integral of the cell density $n_j$ over the spatial domain for $j=2,\ldots,I$, which remains constant over time, while $p(0)$ is approximated as the numerical value of the cell pressure $p$ at the right endpoint of the support of the cell density $n_{1}$ at a time $t$ large enough that the travelling wave is established. We find that there is good agreement between the values of $c_{\text{IBn}}$, $c_{\text{PDEn}}$, and $c_{\text{a}}$ (cf. the values provided in the captions of Figures~\ref{fig:Fig3}-\ref{fig:Fig4} and Figures~\ref{fig:Fig5}-\ref{fig:Fig6}). {\textcolor{black}{We also verified that, when the travelling wave is fully formed, the positions of the right endpoints of the supports of the cell densities are consistent with those predicted by the travelling wave analysis (i.e. those given by \eqref{eq:revrev1}-\eqref{eq:revrev3}). In more detail, denoting by $X_i$ the position of the right endpoint of the support of the cell density $n_i$ at time $t$ large enough that the travelling wave is established, recalling that $z_1 := 0$ (cf. conditions \eqref{ass:SuppTW2}), we verified that the values of $Z_i := X_i - X_1$ for $i =2, \ldots, I$ are consistent with the analytically predicted values, denoted by $Z_{\text{a}i}$, which are obtained by substituting into \eqref{eq:revrev1}-\eqref{eq:revrev3} the analytically predicted wave speed, $c_{\text{a}}$, and the numerically estimated values of $M_j$ for $j=2,\ldots,I$.}} 

\paragraph{Impact of the parameters $\omega_i$ on the shape of travelling fronts} 
The numerical results in Figures~\ref{fig:Fig5} and \ref{fig:Fig6} show that, in agreement with the analytical results of Theorem~\ref{th:theo1}, the choice of the values of the parameters $\omega_i$ impacts on the shape of the travelling fronts that emerge. In more detail, these numerical results indicate that, once the travelling front is established, the cell densities $n_i$ are such that the interface conditions~\eqref{eq:sgndiffn} are satisfied throughout the front. Hence, at the interface between the region occupied by cells with phenotype labelled by the index $i+1$ and the region occupied by cells with phenotype labelled by the index $i$ (i.e. at the interface between the supports of $n_{i+1}$ and $n_{i}$), for $i=1, \ldots, I-1$: if $\omega_{i+1} > \omega_i$ then the value of the cell density at the right of the interface is smaller than the one at the left (cf. top, right panels); if $\omega_{i+1} = \omega_i$ then the value of the cell density at the right of the interface is the same as the one at the left (cf. central, right panels); if $\omega_{i+1} < \omega_i$ then the value of the cell density at the right of the interface is larger than the one at the left (cf. bottom, right panels).

\begin{figure}[H]
\includegraphics[width=\textwidth,trim={4cm 0cm 4cm 0cm},clip]{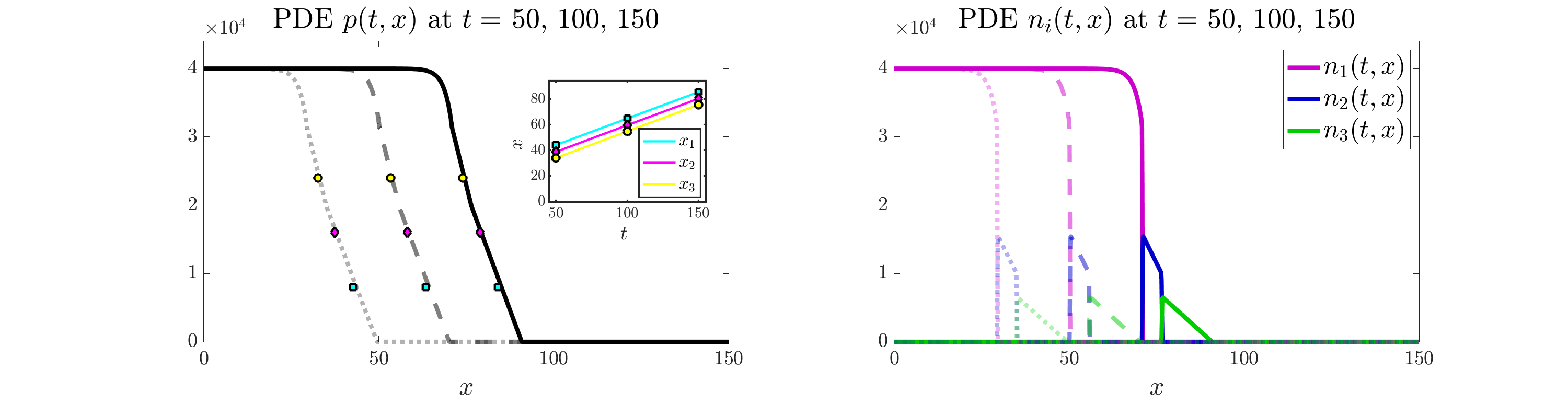}\\
\includegraphics[width=\textwidth,trim={4cm 0cm 4cm 0cm},clip]{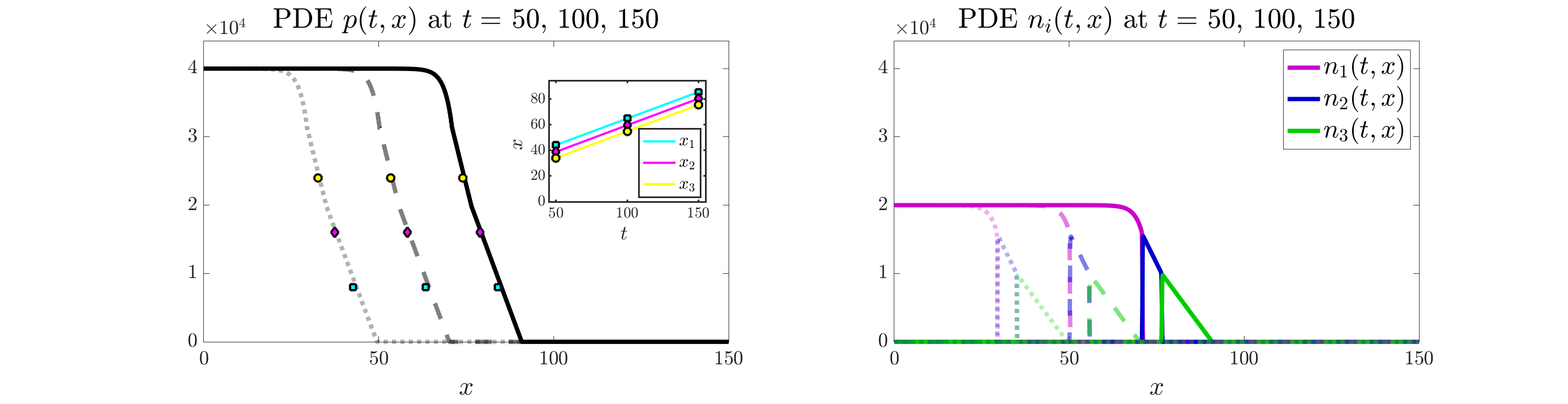}\\
\includegraphics[width=\textwidth,trim={4cm 0cm 4cm 0cm},clip]{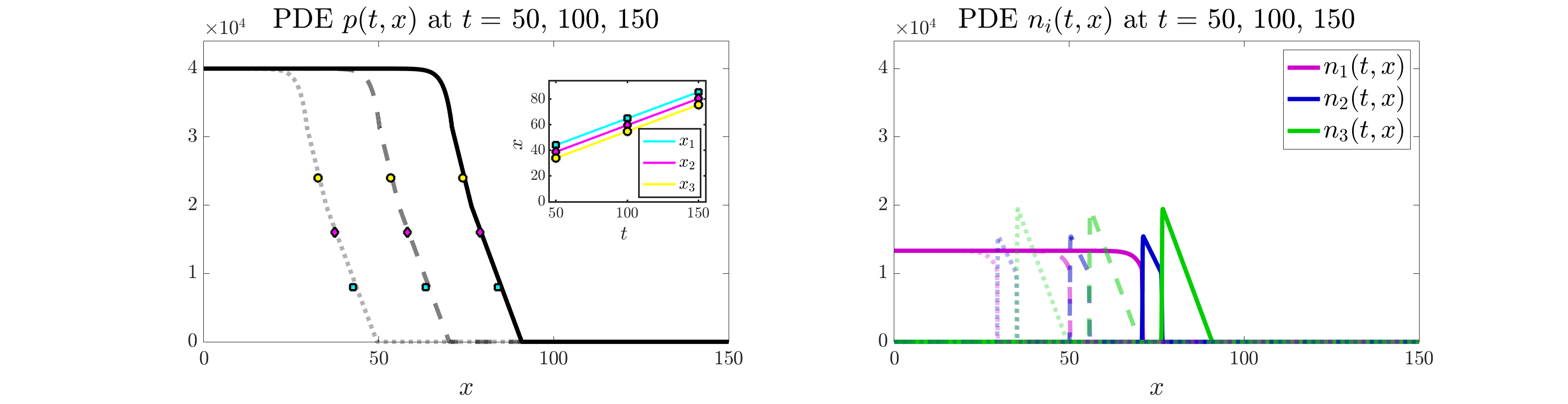}
\caption{{{\bf Main results under different values of the parameters $\omega_i$ for $I=3$.}} Numerical solutions of the continuum model when $I=3$ and the values of the parameters $\alpha_i$ and $\mu_i$ are set according to~\eqref{eq:parametersI=3}, while the values of the parameters $\omega_i$ are: ${\omega_1=1,\ \omega_2=2,\ \omega_3=3}$ {(\bf{\textit{top panels}}}); ${\omega_1=2,\ \omega_2=2,\ \omega_3=2}$ {(\bf{\textit{central panels}}}); and ${\omega_1=3,\ \omega_2=2,\ \omega_3=1}$ {(\bf{\textit{bottom panels}}}). Plots display the cell pressure $p(t,x)$ {(\bf{\textit{left panels}}}) and the cell densities $n_i(t,x)$ {(\bf{\textit{right panels}}}) at three successive time instants -- i.e. $t=50$ \textcolor{black}{(dotted lines)}, $t=100$ \textcolor{black}{(dashed lines)}, and $t=150$ \textcolor{black}{(solid lines)}. The insets of the left panels display the plots of  $x_1(t)$ (cyan), $x_2(t)$ (magenta), and $x_3(t)$ (yellow) defined via~\eqref{eq:xplots}. The coloured markers in the plot of $p(t,x)$ highlight the values of $p(t,x_1(t))$ (cyan), $p(t,x_2(t))$ (magenta), and $p(t,x_3(t))$ (yellow) at $t=50$, $t=100$, and $t=150$. For all cases we report on in this figure, the numerically estimated wave speed is $c_{\text{PDEn}}= 0.42$, while the analytically predicted wave speed is $c_{\text{a}}=0.42$}
\label{fig:Fig5}
\end{figure} 

\begin{figure}[H]
\includegraphics[width=\textwidth,trim={4cm 0cm 4cm 0cm},clip]{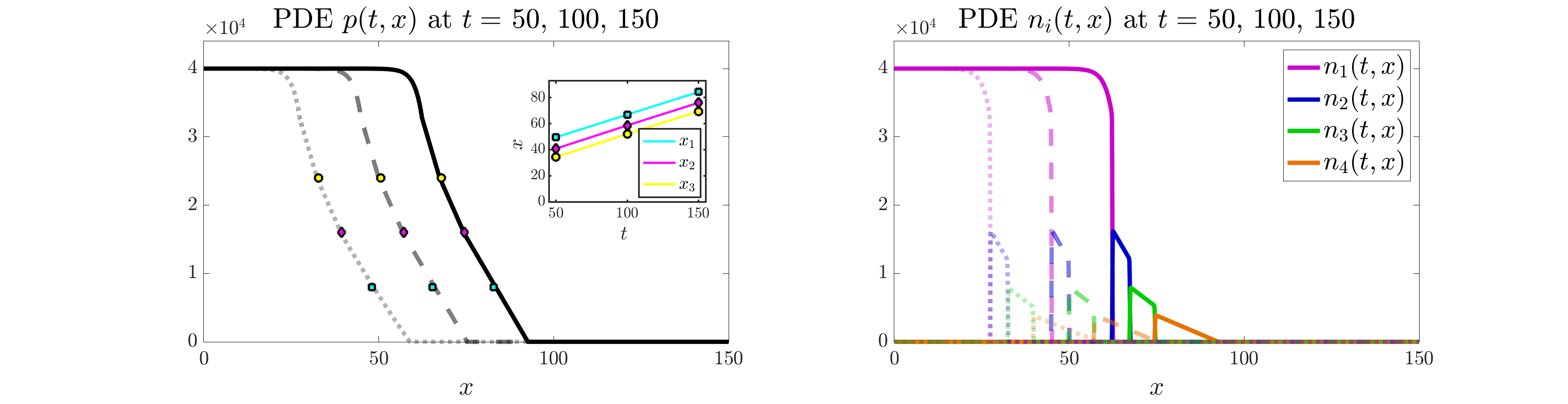}\\
\includegraphics[width=\textwidth,trim={4cm 0cm 4cm 0cm},clip]{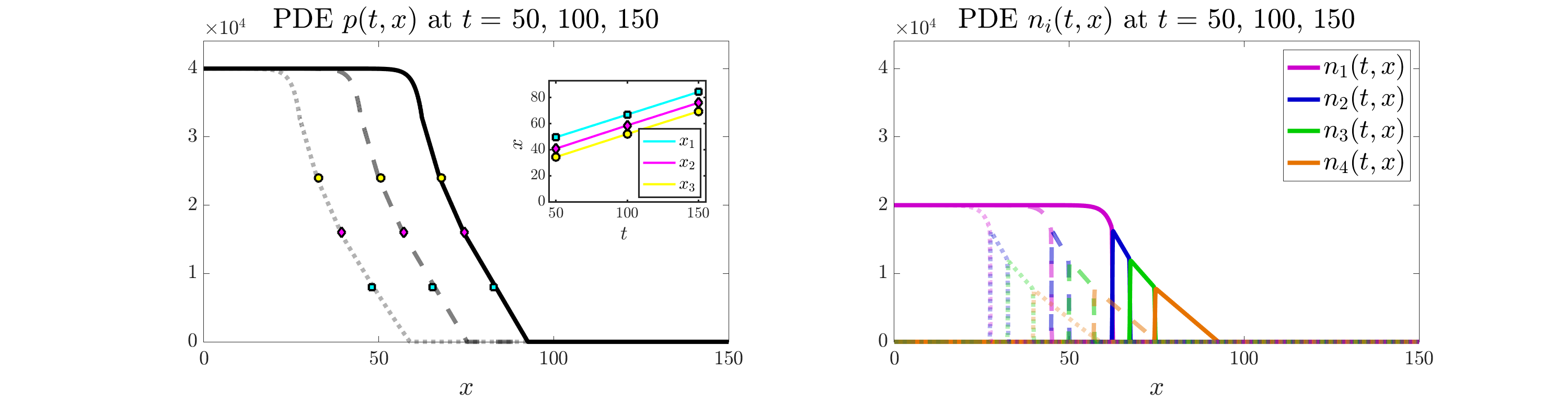}\\
\includegraphics[width=\textwidth,trim={4cm 0cm 4cm 0cm},clip]{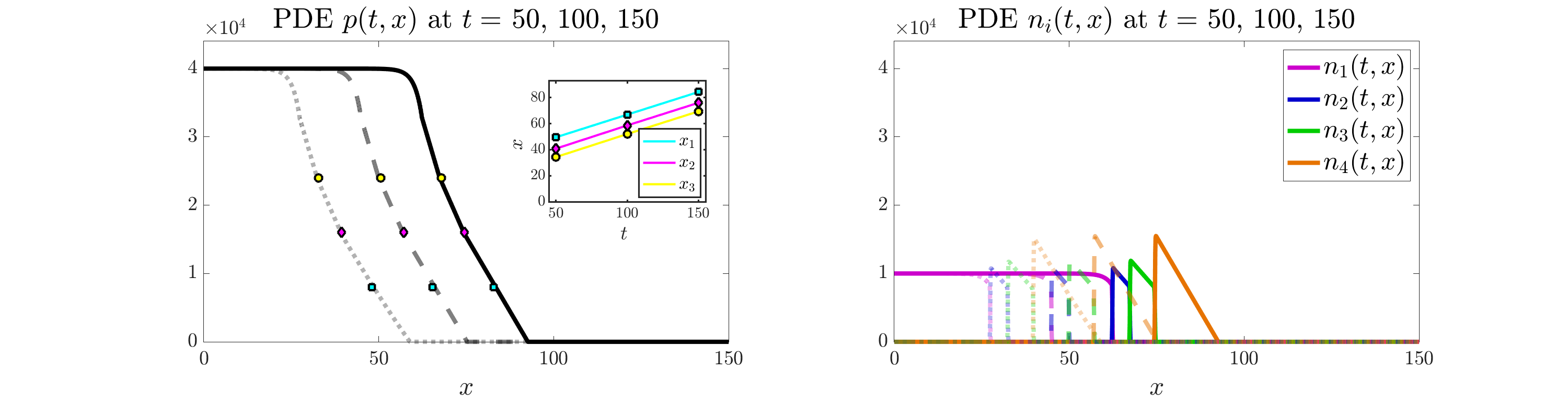}
\caption{{{\bf Main results under different values of the parameters $\omega_i$ for $I=4$.}} Numerical solutions of the continuum model when $I=4$ and the values of the parameters $\alpha_i$ and $\mu_i$ are set according to~\eqref{eq:parametersI=3}-\eqref{eq:parametersI=4}, while the values of the parameters $\omega_i$ are: ${\omega_1=1,\ \omega_2=2,\ \omega_3=3,\ \omega_4=4}$ {(\bf{\textit{top panels}}}); ${\omega_1=2,\ \omega_2=2,\ \omega_3=2,\ \omega_4=2}$ {(\bf{\textit{central panels}}}); and ${\omega_1=4,\ \omega_2=3,\ \omega_3=2,\ \omega_4=1}$ {(\bf{\textit{bottom panels}}}). Plots display the cell pressure $p(t,x)$ {(\bf{\textit{left panels}}}) and the cell densities $n_i(t,x)$ {(\bf{\textit{right panels}}}) at three successive time instants -- i.e. $t=50$ \textcolor{black}{(dotted lines)}, $t=100$ \textcolor{black}{(dashed lines)}, and $t=150$ \textcolor{black}{(solid lines)}. The insets of the left panels display the plots of  $x_1(t)$ (cyan), $x_2(t)$ (magenta), and $x_3(t)$ (yellow) defined via~\eqref{eq:xplots}. The coloured markers in the plot of $p(t,x)$ highlight the values of $p(t,x_1(t))$ (cyan), $p(t,x_2(t))$ (magenta), and $p(t,x_3(t))$ (yellow) at $t=50$, $t=100$, and $t=150$. For all cases we report on in this figure, the numerically estimated wave speed is $c_{\text{PDEn}}= 0.35$, while the analytically predicted wave speed is $c_{\text{a}}=0.35$} 
\label{fig:Fig6}
\end{figure}

\newpage
\section{Discussion and research perspectives}\label{sec:section6}
In this work, we have considered a PDE model for the growth of heterogeneous cell populations subdivided into multiple distinct discrete phenotypes. In this model, cells preferentially move towards regions where they are less compressed, and thus their movement occurs down the gradient of the cellular pressure. The cellular pressure  is defined as a weighted sum of the densities (i.e. the volume fractions) of cells with different phenotypes. To translate into mathematical terms the idea that cells with different phenotypes have different morphological and mechanical properties, both the cell mobility and the weighted amount the cells contribute to the cellular pressure vary with their phenotype. We have formally derived this model as the continuum limit of an on-lattice individual-based model, where cells are represented as single agents undergoing a branching biased random walk corresponding to  phenotype-dependent and pressure-regulated cell division, death, and movement. \textcolor{black}{Then, we have studied travelling wave solutions whereby cells with different phenotypes are spatially separated by sharp boundaries across the invading front (cf. Theorem~\ref{th:theo1}, Remark~\ref{rem:rem2}, and the schematic in Figure~\ref{fig:Fig2}). As discussed in~\cite{batlle2012molecular}, sharp borders between distinct cell types form at the interface of both adjacent tissues and regional domains within a tissue, and the occurrence of sharp spatial segregation between cells is observed both in {\it Drosophila}~\cite{irvine2001boundaries} and in vertebrate tissues~\cite{fraser1990segmentation,langenberg2005lineage,zeltser2001new}.} Finally, we have reported on  numerical simulations of the two models, demonstrating excellent agreement between them and the travelling wave analysis, thus validating the formal limiting procedure employed to derive the individual-based model from the continuum model and confirming the analytical results obtained. 

\textcolor{black}{The results presented here indicate that inter-cellular variability in mobility can support the maintenance of spatial segregation across invading fronts, whereby cells with a higher mobility drive invasion by occupying regions closer to the front edge.} These results have been obtained under the assumption that cells with phenotypes labelled by larger values of the index $i$ express a higher mobility -- cf. assumptions~\eqref{eq:gamma_conditions} on the parameters $\gamma_i$ of the individual-based model and the corresponding assumptions~\eqref{ass:modpar3} on the parameters $\mu_i$ of the continuum model. On the other hand, no specific assumptions have been made on the parameters $\omega_i$, which provide a measure of the weighted amount that cells with phenotype $i$ contribute towards the cellular pressure (cf. the constitutive relations~\eqref{eq:PDEmodel_3} and \eqref{eq:pressure_definition}) and the values of which can be related, for instance, to cell stiffness -- i.e. if cells with phenotype labelled by the index $i$ are stiffer than cells with phenotype labelled by the index $j$ then $\omega_i>\omega_j$. Therefore, the results we have presented apply both to scenarios where less stiff cells are more invasive and to opposite scenarios -- scenarios that, under assumptions~\eqref{eq:gamma_conditions} and~\eqref{ass:modpar3}, would correspond to assuming, respectively, $\omega_{i+1}<\omega_i$ and $\omega_{i+1}>\omega_i$ for $i=1,\ldots,I-1$. This is particularly relevant in the context of tumour growth. In fact, it has been observed, through both {\textit{in vitro}} and {\textit{in vivo}} experiments, that solid tumours can be made up of cells of varying stiffness~\cite{lv2021cell,swaminathan2011mechanical,rianna2020direct,baker2010cancer,han2020cell}. In a large number of cancer cell lines, more migratory and invasive phenotypes align with the cells that are softest/most deformable~\cite{lv2021cell,han2020cell}. For example, in ovarian cancer cell lines, cells with the highest migration capabilities can be up to ten times less stiff than cells with the lowest migration and invasion potential~\cite{swaminathan2011mechanical}. However, in some cell lines, such as breast cancer cell lines, stiffness has been observed to increase with tumourigenic invasive potential, especially in areas of stiff extracellular matrix~\cite{baker2010cancer,mok2019mapping}. 

We conclude with an outlook on possible research perspectives. \textcolor{black}{In this work the focus has been placed on the {\it propagation} of segregation properties, and we have thus studied the existence of travelling wave solutions of the system of PDEs~\eqref{eq:PDEmodel} that exhibit segregation between different cell types. Accordingly, we have carried out numerical simulations under initial conditions corresponding to scenarios where cells of different types initially occupy distinct regions of the spatial domain. As the next step, it would be relevant to investigate the {\it emergence} of segregation properties by studying analytically the convergence of solutions of the PDE system~\eqref{eq:PDEmodel} to such travelling wave solutions, and carrying out numerical simulations in cases where cells of different types are not separated at the initial time.}

\textcolor{black}{It would also be interesting to explore scenarios where assumptions~\eqref{eq:alpha_conditions} on the parameters $\alpha_i$ are relaxed (i.e. when also proliferation and death of cells with phenotypes labelled by values of the index $i > 1$ are incorporated into the model). In this regard, under assumptions~\eqref{ass:modpar3} on the mobility coefficients $\mu_i$, taking into account proliferation-migration trade-offs induced by the inherent energetic cost attached to cellular activities, it would be natural to consider the variant~\eqref{eq:PDEmodelrev} of the PDE system~\eqref{eq:PDEmodel} subject to assumptions~\eqref{eq:G_conditions} along with the assumptions $\alpha_1 > \alpha_2 > \ldots > \alpha_{I-1} > \alpha_{I} = 0$. In this case, the calculations carried out  in {\it Step 4} of the proof of Theorem~\ref{th:theo1} would break down, hinting that conditions~\eqref{eq:kinkPns}, which express the fact that the interfaces between the components $n_{i}(z)$ and $n_{i+1}(z)$ of the solution travel at the same speed $c$ for all $i=1,\ldots,I-1$, would not hold. Hence, under this scenario it is likely that the continuum model does not admit travelling wave solutions of the type of those of Theorem~\ref{th:theo1}.} 

While, focusing our attention on spatial segregation across travelling fronts, here we have considered the case where assumptions~\eqref{ass:modpar3} hold, this work could be extended by investigating the behaviour of solutions to the PDE system~\eqref{eq:PDEmodel}, and its more general variant~\eqref{eq:PDEmodelrev}, in cases where these assumptions do not hold and segregation properties may not be propagated~\cite{david2024degenerate,lorenzi2017interfaces,carrillo2018zoology}. It would also be interesting to study free boundary problems for these PDE systems, along the lines of those considered in~\cite{byrne1997free,lorenzi2020individual,mimura2010free} to model tissue development and tumour growth, and consider related transmission problems modelling cell invasion through thin membranes, in the vein of~\cite{chaplain2019derivation,ciavolella2024effective,ciavolella2021existence,giverso2022effective}.

Moreover, we could incorporate into the PDE system~\eqref{eq:PDEmodel}, or its generalised variant~\eqref{eq:PDEmodelrev}, the effects of phenotypic switching, possibly driven by the extracellular \textcolor{black}{environment~\cite{celora2021phenotypic,charras2014physical}}, and cell-cell adhesion, as similarly done for instance in~\cite{bubba2020hele,carrillo2018splitting,crossley2024phenotypic,macfarlane2022impact} and~\cite{berendsen2017cross,burger2020segregation,carrillo2019population,carrillo2018zoology,carrillo2018splitting}, respectively. Instead of resorting to phenomenological considerations to define the additional terms modelling these phenomena, we could extend the individual-based model considered here along with the formal procedure employed to derive the corresponding continuum model so as to encompass mechanisms of phenotypic switching and cell-cell adhesion. This would be biologically interesting in that it has been shown, on the one hand, that some cancer cell lines can adaptively alter their stiffness to become softer in order to overcome physical barriers or areas of high cellular pressure~\cite{rianna2020direct,han2020cell}, and, on the other hand, that in general internal stiffness of aggressive tumours is more heterogeneous than in quiescent tumours~\cite{mok2019mapping}. Adapting our current model to include dynamical variation in cell stiffness through the weights modelling contribution to the cellular pressure could be of interest to study these dynamics further. On this note, recent work by Zills {\textit{et. al.}}~\cite{zills2023enhanced} considered a 3D individual-based model of a growing cell population where cells could become softer during cell division. The cell stiffness then related to adhesive and repulsive forces of the cells and their contribution to the local cellular pressure, which in turn regulated cell division rates. Numerical simulations of this model were able to replicate experimental data on cellular spheroids where cells towards the border were faster, larger, and softer than those at the centre of the spheroid. However, as the model considered is an individual-based model, the study in~\cite{zills2023enhanced} is based on numerical simulations only. Hence, it would be interesting to extend our modelling framework to include such additional aspects and then carry out travelling wave analysis, in order to complement numerical simulations with analytical results to facilitate a more comprehensive exploration of the model parameter space.

\textcolor{black}{Another avenue for future research could be to investigate spatial segregation across travelling fronts in individual-based and continuum models for the growth of heterogeneous cell populations that encapsulate volume exclusion effects, which are not captured by the models considered here. For this, we expect modelling methods and analytical techniques similar to those employed in~\cite{baker2019free,fadai2020new,murphy2020mechanical,murphy2021travelling,tambyah2020free} to be useful.}

As a further extension of the present work, building on the modelling approach presented in~\cite{david2023phenotypic,lorenzi2023derivation,lorenzi2022trade,lorenzi2021}, we could also let the cell phenotype vary along a spectrum, and thus be described by a continuous variable $y \in \mathbb{R}$~\cite{lorenzi2024phenotype}. In this case, the evolution of the density (i.e. the volume fraction) of cells with phenotype $y$ at time $t \geq 0$, $n(t,x,y)$, would be governed by the following partial integro-differential equation model
\begin{equation}
\label{eq:PDEmodelrevcont}
\begin{cases}
\begin{array}{l}
\displaystyle{\partial_t n - \mu(y) \, \partial_x \left(n \, \partial_x p \right) = \alpha(y) \, G(p) \, n, \quad y \in \mathbb{R}},
\\\\
\displaystyle{p(t,x) := \int_{\mathbb{R}} \omega(y) \, n(t,x,y) \, {\rm d}y,}
\end{array}
\quad (t,x) \in (0,\infty) \times \mathbb{R}.
\end{cases}
\end{equation}
Note that, compared to the generalised form~\eqref{eq:PDEmodelrev} of the PDE model~\eqref{eq:PDEmodel}, subject to assumptions~\eqref{eq:G_conditions}, here the parameters $\mu_i$, $\alpha_i$, and $\omega_i$ have been replaced by the functions $\mu(y)$, $\alpha(y)$, and $\omega(y)$, respectively. We expect the model~\eqref{eq:PDEmodelrevcont} to be derivable from an underlying individual-based model through a formal limiting procedure analogous to the one employed in~\cite{lorenzi2023derivation,macfarlane2022individual}. 

The aforementioned extensions of the present work, which will bring new mathematical problems, will allow for further investigation into how phenotypic heterogeneity shapes invading fronts in growing cell populations.

\backmatter

\bmhead{Acknowledgements}
JAC was supported by the Advanced Grant Nonlocal-CPD (Nonlocal PDEs for Complex Particle Dynamics: Phase Transitions, Patterns and Synchronization) of the European Research Council Executive Agency (ERC) under the European Union’s Horizon 2020 research and innovation programme (grant agreement No. 883363). JAC was also partially supported by the EPSRC grant number EP/V051121/1 and by the “Maria de Maeztu” Excellence Unit IMAG, reference CEX2020-001105-M, funded by MCIN/AEI/10.13039/501100011033/. 
TL gratefully acknowledges support from: the Italian Ministry of University and Research (MUR) through
the grant PRIN 2020 project (No. 2020JLWP23) ``Integrated Mathematical Approaches to Socio-Epidemiological Dynamics'' (CUP: E15F21005420006) and the grant PRIN2022-PNRR project (No. P2022Z7ZAJ) ``A Unitary Mathematical Framework for Modelling Muscular Dystrophies'' (CUP: E53D23018070001) funded by the European Union - Next Generation EU; and the Istituto Nazionale di Alta Matematica (INdAM) and the Gruppo Nazionale per la Fisica Matematica (GNFM). TL would also like to thank Luigi Preziosi for insightful discussions during the development of the project.

\section*{Data Availability}
The code used for numerical simulations and data for this manuscript is available upon request.

\begin{appendices}

\section{Formal derivation of the continuum model from the individual-based model}\label{app:appendixA}
We detail the formal derivation of the PDE system~\eqref{eq:PDEmodel} from the branching biased random walk underlying the individual-based model described in Section~\ref{sec:section2}.

When cell dynamics are governed by the rules described in Section~\ref{sec:ibmodn}, using that $1+\tau G_i(p_j^k)_+-\tau G_i(p_j^k)_-=1+\tau G_i(p_j^k)$, the principle of mass balance gives the balance equation~\eqref{eq:mastereq}. Then, employing a formal procedure analogous to the one that we used in \cite{chaplain2020bridging,macfarlane2020hybrid,macfarlane2022individual}, we define
\[
n_i \equiv n_i(t,x) := n_{i,\ j}^{k},\quad n_i(t+\tau,x) := n_{i,\ j}^{k+1},\quad n_i(t,x\pm\Delta_x) := n_{i,\ j\pm1}^{k}
\]
and
\[
p \equiv p(t,x) := p_{j}^{k}, \quad p(t,x\pm \Delta_x) := p_{j\pm 1}^{k}.
\]
From the constitutive relation~\eqref{eq:pressure_definition} we find
\begin{equation}
p(t,x):=\displaystyle{\sum_{i=1}^{I} \omega_i \ n_i(t,x)}, \label{eq:pressure_definitionderiv}
\end{equation} 
while from the balance equation~\eqref{eq:mastereq} we obtain
\begin{eqnarray*}
n_i(t+\tau,x)&=&n_i\left\{\left(1+\tau G_i(p)\right)\left[1-\frac{\gamma_i \left(p-p(t,x+\Delta_x)\right)_+}{2\overline{p}}-\frac{\gamma_i \left(p-p(t,x-\Delta_x)\right)_+}{2\overline{p}}\right]\right\}\nonumber\\
&&+n_i(t,x+\Delta_x)\left\{ \left(1+\tau G_i(p(t,x+\Delta_x))\right)\left[\frac{\gamma_i \left(p(t,x+\Delta_x)-p\right)_+}{2\overline{p}}\right]\right\}\\
&&+n_i(t,x-\Delta_x)\left\{ \left(1+\tau G_i(p(t,x-\Delta_x))\right)\left[\frac{\gamma_i \left(p(t,x-\Delta_x)-p\right)_+}{2\overline{p}}\right]\right\}.\nonumber
\end{eqnarray*}
Splitting the growth terms gives
\begin{eqnarray*}
n_i(t+\tau,x)&=&n_i\left[1-\frac{\gamma_i \left(p-p(t,x+\Delta_x)\right)_+}{2\overline{p}}-\frac{\gamma_i \left(p-p(t,x-\Delta_x)\right)_+}{2\overline{p}}\right]\nonumber\\
&&+n_i(t,x+\Delta_x)\left[\frac{\gamma_i \left(p(t,x+\Delta_x)-p\right)_+}{2\overline{p}}\right]\nonumber\\
&&+n_i(t,x-\Delta_x)\left[ \frac{\gamma_i \left(p(t,x-\Delta_x)-p\right)_+}{2\overline{p}}\right]\\
&&+\tau G_i(p)n_i\left[ 1-\frac{\gamma_i \left(p-p(t,x+\Delta_x)\right)_+}{2\overline{p}}-\frac{\gamma_i \left(p-p(t,x-\Delta_x)\right)_+}{2\overline{p}}\right]\nonumber\\
&&+\tau G_i(p(t,x+\Delta_x))n_i(t,x+\Delta_x)\left[\frac{\gamma_i \left(p(t,x+\Delta_x)-p\right)_+}{2\overline{p}}\right]\nonumber\\
&&+\tau G_i(p(t,x-\Delta_x))n_i(t,x-\Delta_x)\left[\frac{\gamma_i \left(p(t,x-\Delta_x)-p\right)_+}{2\overline{p}}\right].\nonumber
\end{eqnarray*}
Now, assuming $n_i$ to be sufficiently regular, using that
\[
n_i(t,x\pm\Delta_x)= n_i \pm \Delta_x \frac{\partial n_i}{\partial x}+\frac{\Delta_x^2}{2}\frac{\partial ^2 n_i}{\partial x^2 }+\mathcal{O}(\Delta_x^3),
\]
we formally obtain
\begin{eqnarray*}
n_i(t+\tau,x)&=&n_i\left[1-\frac{\gamma_i \left(p-p(t,x+\Delta_x)\right)_+}{2\overline{p}}-\frac{\gamma_i \left(p-p(t,x-\Delta_x)\right)_+}{2\overline{p}}\right]\nonumber\\
&&+\left(n_i+\Delta_x \frac{\partial n_i}{\partial x}+\frac{\Delta_x^2}{2}\frac{\partial ^2 n_i}{\partial x^2 }\right)\left[\frac{\gamma_i \left(p(t,x+\Delta_x)-p\right)_+}{2\overline{p}}\right]\nonumber\\
&&+\left(n_i-\Delta_x \frac{\partial n_i}{\partial x}+\frac{\Delta_x^2}{2}\frac{\partial ^2 n_i}{\partial x^2 }\right)\left[ \frac{\gamma_i \left(p(t,x-\Delta_x)-p\right)_+}{2\overline{p}}\right]\\
&&+\tau G_i(p)n_i\left[ 1-\frac{\gamma_i \left(p-p(t,x+\Delta_x)\right)_+}{2\overline{p}}-\frac{\gamma_i \left(p-p(t,x-\Delta_x)\right)_+}{2\overline{p}}\right]\nonumber\\
&&+\tau G_i(p(t,x+\Delta_x))\left(n_i+\Delta_x \frac{\partial n_i}{\partial x}+\frac{\Delta_x^2}{2}\frac{\partial ^2 n_i}{\partial x^2 }\right)\left[\frac{\gamma_i \left(p(t,x+\Delta_x)-p\right)_+}{2\overline{p}}\right]\nonumber\\
&&+\tau G_i(p(t,x-\Delta_x))\left(n_i-\Delta_x \frac{\partial n_i}{\partial x}+\frac{\Delta_x^2}{2}\frac{\partial ^2 n_i}{\partial x^2 }\right)\left[\frac{\gamma_i \left(p(t,x-\Delta_x)-p\right)_+}{2\overline{p}}\right]+\mathcal{O}(\Delta_x^3).\nonumber
\end{eqnarray*}
Furthermore, rearranging terms yields
{
\begin{eqnarray*}
n_i(t+\tau,x)&=&n_i\left[1-\frac{\gamma_i \left(p-p(t,x+\Delta_x)\right)_+}{2\overline{p}}-\frac{\gamma_i \left(p-p(t,x-\Delta_x)\right)_+}{2\overline{p}}\right]\nonumber\\
&&+n_i\left[\frac{\gamma_i \left(p(t,x+\Delta_x)-p\right)_+}{2\overline{p}}+\frac{\gamma_i \left(p(t,x-\Delta_x)-p\right)_+}{2\overline{p}}\right]\nonumber\\
&&+\Delta_x \frac{\partial n_i}{\partial x}\left[\frac{\gamma_i \left(p(t,x+\Delta_x)-p\right)_+}{2\overline{p}}- \frac{\gamma_i \left(p(t,x-\Delta_x)-p\right)_+}{2\overline{p}}\right]\nonumber\\
&&+\frac{\Delta_x^2}{2}\frac{\partial ^2 n_i}{\partial x^2 }\left[\frac{\gamma_i \left(p(t,x+\Delta_x)-p\right)_+}{2\overline{p}}+ \frac{\gamma_i \left(p(t,x-\Delta_x)-p\right)_+}{2\overline{p}}\right]\\
&&+\tau G_i(p)n_i\left[1-\frac{\gamma_i \left(p-p(t,x+\Delta_x)\right)_+}{2\overline{p}}-\frac{\gamma_i \left(p-p(t,x-\Delta_x)\right)_+}{2\overline{p}}\right]\nonumber\\
&&+\tau G_i(p)n_i\left[\frac{\gamma_i \left(p(t,x+\Delta_x)-p\right)_+}{2\overline{p}}+\frac{\gamma_i \left(p(t,x-\Delta_x)-p\right)_+}{2\overline{p}}\right]\nonumber\\
&&+\Delta_x \frac{\partial n_i}{\partial x}\left[\tau G_i(p(t,x+\Delta_x))\frac{\gamma_i \left(p(t,x+\Delta_x)-p\right)_+}{2\overline{p}}\right]\nonumber\\
&&-\Delta_x \frac{\partial n_i}{\partial x}\left[\tau G_i(p(t,x-\Delta_x))\frac{\gamma_i \left(p(t,x-\Delta_x)-p\right)_+}{2\overline{p}}\right]\nonumber\\
&&+\frac{\Delta_x^2}{2} \frac{\partial^2 n_i}{\partial x^2}\left[\tau G_i(p(t,x+\Delta_x))\frac{\gamma_i \left(p(t,x+\Delta_x)-p\right)_+}{2\overline{p}}\right]\nonumber\\
&&+\frac{\Delta_x^2}{2} \frac{\partial^2 n_i}{\partial x^2}\left[\tau G_i(p(t,x-\Delta_x))\frac{\gamma_i \left(p(t,x-\Delta_x)-p\right)_+}{2\overline{p}}\right]+\mathcal{O}(\Delta_x^3).\nonumber
\end{eqnarray*}
}
Then, using the property that $(f)_+ -(-f)_+=(f)_+ -(f)_-=f$, we can simplify the latter equation to
\begin{eqnarray*}
n_i(t+\tau,x)&=&n_i\left[1+\frac{\gamma_i \left(p(t,x+\Delta_x)-p\right)}{2\overline{p}}+\frac{\gamma_i \left(p(t,x-\Delta_x)-p\right)}{2\overline{p}}\right]\nonumber\\
&&+\Delta_x \frac{\partial n_i}{\partial x}\left[\frac{\gamma_i \left(p(t,x+\Delta_x)-p\right)_+}{2\overline{p}}- \frac{\gamma_i \left(p(t,x-\Delta_x)-p\right)_+}{2\overline{p}}\right]\nonumber\\
&&+\frac{\Delta_x^2}{2}\frac{\partial ^2 n_i}{\partial x^2 }\left[\frac{\gamma_i \left(p(t,x+\Delta_x)-p\right)_+}{2\overline{p}}+ \frac{\gamma_i \left(p(t,x-\Delta_x)-p\right)_+}{2\overline{p}}\right]\\
&&+\tau G_i(p)n_i\left[1+\frac{\gamma_i \left(p(t,x+\Delta_x)-p\right)}{2\overline{p}}+\frac{\gamma_i \left(p(t,x-\Delta_x)-p\right)}{2\overline{p}}\right]\nonumber\\
&&+\Delta_x \frac{\partial n_i}{\partial x}\left[\tau G_i(p(t,x+\Delta_x))\frac{\gamma_i \left(p(t,x+\Delta_x)-p\right)_+}{2\overline{p}}\right]\nonumber\\
&&-\Delta_x \frac{\partial n_i}{\partial x}\left[\tau G_i(p(t,x-\Delta_x))\frac{\gamma_i \left(p(t,x-\Delta_x)-p\right)_+}{2\overline{p}}\right]\nonumber\\
&&+\frac{\Delta_x^2}{2} \frac{\partial^2 n_i}{\partial x^2}\left[\tau G_i(p(t,x+\Delta_x))\frac{\gamma_i \left(p(t,x+\Delta_x)-p\right)_+}{2\overline{p}}\right]\nonumber\\
&&+\frac{\Delta_x^2}{2} \frac{\partial^2 n_i}{\partial x^2}\left[\tau G_i(p(t,x-\Delta_x))\frac{\gamma_i \left(p(t,x-\Delta_x)-p\right)_+}{2\overline{p}}\right]\nonumber\\
&&+\mathcal{O}(\Delta_x^3),\nonumber
\end{eqnarray*}
from which, if $p$ is sufficiently regular, using that

\[
p(t,x\pm\Delta_x)=p\pm \Delta_x \frac{\partial p}{\partial x}+\frac{\Delta_x^2}{2}\frac{\partial^2 p}{\partial x^2}+\mathcal{O}(\Delta_x^3),
\]
we formally obtain
{{
\begin{eqnarray*}
n_i(t+\tau,x)&=&n_i\left[1+\frac{\gamma_i \left(\Delta_x \frac{\partial p}{\partial x}+\frac{\Delta_x^2}{2}\frac{\partial^2 p}{\partial x^2}\right)}{2\overline{p}}+\frac{\gamma_i \left(-\Delta_x \frac{\partial p}{\partial x}+\frac{\Delta_x^2}{2}\frac{\partial^2 p}{\partial x^2}\right)}{2\overline{p}}\right]\nonumber\\
&&+\Delta_x \frac{\partial n_i}{\partial x}\left[\frac{\gamma_i \left(\Delta_x \frac{\partial p}{\partial x}+\frac{\Delta_x^2}{2}\frac{\partial^2 p}{\partial x^2}\right)_+}{2\overline{p}}- \frac{\gamma_i \left(-\Delta_x \frac{\partial p}{\partial x}+\frac{\Delta_x^2}{2}\frac{\partial^2 p}{\partial x^2}\right)_+}{2\overline{p}}\right]\nonumber\\
&&+\frac{\Delta_x^2}{2}\frac{\partial ^2 n_i}{\partial x^2 }\left[\frac{\gamma_i \left(\Delta_x \frac{\partial p}{\partial x}+\frac{\Delta_x^2}{2}\frac{\partial^2 p}{\partial x^2}\right)_+}{2\overline{p}}+ \frac{\gamma_i \left(-\Delta_x \frac{\partial p}{\partial x}+\frac{\Delta_x^2}{2}\frac{\partial^2 p}{\partial x^2}\right)_+}{2\overline{p}}\right]\\
&&+\tau G_i(p)n_i\left[1+\frac{\gamma_i \left(\Delta_x \frac{\partial p}{\partial x}+\frac{\Delta_x^2}{2}\frac{\partial^2 p}{\partial x^2}\right)}{2\overline{p}}+\frac{\gamma_i \left(-\Delta_x \frac{\partial p}{\partial x}+\frac{\Delta_x^2}{2}\frac{\partial^2 p}{\partial x^2}\right)}{2\overline{p}}\right]\nonumber\\
&&+\Delta_x \frac{\partial n_i}{\partial x}\left[\tau G_i\left(p+\Delta_x \frac{\partial p}{\partial x}+\frac{\Delta_x^2}{2}\frac{\partial^2 p}{\partial x^2}\right)\frac{\gamma_i \left(\Delta_x \frac{\partial p}{\partial x}+\frac{\Delta_x^2}{2}\frac{\partial^2 p}{\partial x^2}\right)_+}{2\overline{p}}\right]\nonumber\\
&&-\Delta_x \frac{\partial n_i}{\partial x}\left[\tau G_i\left(p-\Delta_x \frac{\partial p}{\partial x}+\frac{\Delta_x^2}{2}\frac{\partial^2 p}{\partial x^2}\right)\frac{\gamma_i \left(-\Delta_x \frac{\partial p}{\partial x}+\frac{\Delta_x^2}{2}\frac{\partial^2 p}{\partial x^2}\right)_+}{2\overline{p}}\right]\nonumber\\
&&+\frac{\Delta_x^2}{2} \frac{\partial^2 n_i}{\partial x^2}\left[\tau G_i\left(p+\Delta_x \frac{\partial p}{\partial x}+\frac{\Delta_x^2}{2}\frac{\partial^2 p}{\partial x^2}\right)\frac{\gamma_i \left(\Delta_x \frac{\partial p}{\partial x}+\frac{\Delta_x^2}{2}\frac{\partial^2 p}{\partial x^2}\right)_+}{2\overline{p}}\right]\nonumber\\
&&+\frac{\Delta_x^2}{2} \frac{\partial^2 n_i}{\partial x^2}\left[\tau G_i\left(p-\Delta_x \frac{\partial p}{\partial x}+\frac{\Delta_x^2}{2}\frac{\partial^2 p}{\partial x^2}\right)\frac{\gamma_i \left(-\Delta_x \frac{\partial p}{\partial x}+\frac{\Delta_x^2}{2}\frac{\partial^2 p}{\partial x^2}\right)_+}{2\overline{p}}\right]+\mathcal{O}(\Delta_x^3).\nonumber
\end{eqnarray*}}}
Rearranging terms and neglecting terms of order higher than $\mathcal{O}(\Delta_x^3)$ we find
{\small{
\begin{eqnarray*}
n_i(t+\tau,x)&=&n_i\left[1+\frac{\gamma_i \Delta_x^2}{2\overline{p}} \frac{\partial^2 p}{\partial x^2}\right]+\Delta_x \frac{\partial n_i}{\partial x}\left[\frac{\gamma_i \left(\Delta_x \frac{\partial p}{\partial x}\right)_+}{2\overline{p}}- \frac{\gamma_i \left(-\Delta_x \frac{\partial p}{\partial x}\right)_+}{2\overline{p}}\right]\\
&&+\tau G_i(p)n_i\left[1+\frac{\gamma_i \Delta_x^2}{2\overline{p}} \frac{\partial^2 p}{\partial x^2}\right]\nonumber\\
&&+\Delta_x \frac{\partial n_i}{\partial x}\left[\tau G_i\left(p+\Delta_x \frac{\partial p}{\partial x}+\frac{\Delta_x^2}{2}\frac{\partial^2 p}{\partial x^2}\right)\frac{\gamma_i \left(\Delta_x \frac{\partial p}{\partial x}\right)_+}{2\overline{p}}\right]\nonumber\\
&&-\Delta_x \frac{\partial n_i}{\partial x}\left[\tau G_i\left(p-\Delta_x \frac{\partial p}{\partial x}+\frac{\Delta_x^2}{2}\frac{\partial^2 p}{\partial x^2}\right)\frac{\gamma_i \left(-\Delta_x \frac{\partial p}{\partial x}\right)_+}{2\overline{p}}\right]+\mathcal{O}(\Delta_x^3).\nonumber
\end{eqnarray*}}}
Again, using that $(f)_+ -(-f)_+=(f)_+ -(f)_-=f$, we can simplify the above equation to
\begin{eqnarray*}
n_i(t+\tau,x)&=&n_i\left[1+\frac{\gamma_i \Delta_x^2}{2\overline{p}} \frac{\partial^2 p}{\partial x^2}\right]\nonumber
+\frac{\gamma_i\Delta_x^2}{2\overline{p}} \frac{\partial n_i}{\partial x}\frac{\partial p}{\partial x}
+\tau G_i(p)n_i\left[1+\frac{\gamma_i \Delta_x^2}{2\overline{p}} \frac{\partial^2 p}{\partial x^2}\right]\nonumber\\
&&+\Delta_x \frac{\partial n_i}{\partial x}\left[\tau G_i\left(p+\Delta_x \frac{\partial p}{\partial x}+\frac{\Delta_x^2}{2}\frac{\partial^2 p}{\partial x^2}\right)\frac{\gamma_i \left(\Delta_x \frac{\partial p}{\partial x}\right)_+}{2\overline{p}}\right]\nonumber\\
&&-\Delta_x \frac{\partial n_i}{\partial x}\left[\tau G_i\left(p-\Delta_x \frac{\partial p}{\partial x}+\frac{\Delta_x^2}{2}\frac{\partial^2 p}{\partial x^2}\right)\frac{\gamma_i \left(-\Delta_x \frac{\partial p}{\partial x}\right)_+}{2\overline{p}}\right]+\mathcal{O}(\Delta_x^3),\nonumber
\end{eqnarray*}
from which, rearranging again terms, we obtain
\begin{eqnarray*}
&&n_i(t+\tau,x)=n_i+\frac{\gamma_i \Delta_x^2}{2\overline{p}}\left[n_i \frac{\partial^2 p}{\partial x^2}+ \frac{\partial n_i}{\partial x}\frac{\partial p}{\partial x}\right]+\tau G_i(p)n_i
+\frac{\gamma_i \Delta_x^2 \tau}{2\overline{p}} G_i(p)n_i\frac{\partial^2 p}{\partial x^2}\\
&&+\frac{\gamma_i \Delta_x^2 \tau}{2\overline{p}}\frac{\partial n_i}{\partial x}\left[ \alpha_iG\left(p+\Delta_x \frac{\partial p}{\partial x}+\frac{\Delta_x^2}{2}\frac{\partial^2 p}{\partial x^2}\right)\left(\frac{\partial p}{\partial x}\right)_+-G_i\left(p-\Delta_x \frac{\partial p}{\partial x}+\frac{\Delta_x^2}{2}\frac{\partial^2 p}{\partial x^2}\right) \left(-\frac{\partial p}{\partial x}\right)_+\right]\nonumber\\
&&+\mathcal{O}(\Delta_x^3).\nonumber
\end{eqnarray*}
Now taking $n_i$ over to the left-hand side and dividing through by $\tau$ yields
\begin{eqnarray*}
&&\frac{n_i(t+\tau,x)-n_i}{\tau}=\frac{\gamma_i \Delta_x^2}{2\tau \overline{p}}\left[n_i \frac{\partial^2 p}{\partial x^2}+ \frac{\partial n_i}{\partial x}\frac{\partial p}{\partial x}\right]+ G_i(p)n_i+\frac{\gamma_i \Delta_x^2 }{2\overline{p}} G_i(p)n_i\frac{\partial^2 p}{\partial x^2}\\
&&+\frac{\gamma_i \Delta_x^2 }{2\overline{p}}\frac{\partial n_i}{\partial x}\left[ G_i\left(p+\Delta_x \frac{\partial p}{\partial x}+\frac{\Delta_x^2}{2}\frac{\partial^2 p}{\partial x^2}\right)\left(\frac{\partial p}{\partial x}\right)_+-G_i\left(p-\Delta_x \frac{\partial p}{\partial x}+\frac{\Delta_x^2}{2}\frac{\partial^2 p}{\partial x^2}\right) \left(-\frac{\partial p}{\partial x}\right)_+\right]\nonumber\\
&&+\mathcal{O}(\Delta_x^3).\nonumber
\end{eqnarray*}
Finally, letting $\Delta_x\rightarrow 0$ and $\tau \rightarrow 0$ in such a way that
$$
\frac{\gamma_i \Delta_x^2}{2\tau \overline{p}}\rightarrow \mu_i, \quad \text{where} \quad \mu_i \in \mathbb{R}^*_+,\ i=1, \ldots, I,
$$
we formally obtain the following PDEs
\begin{eqnarray*}
\frac{\partial n_i}{\partial t}&=&\mu_i\left(n_i \frac{\partial^2 p}{\partial x^2}+ \frac{\partial n_i}{\partial x}\frac{\partial p}{\partial x}\right)+ G_i(p)n_i, \quad i=1,\ldots,I, \quad (t,x) \in (0,\infty) \times \mathbb{R}, \nonumber
\end{eqnarray*}
subject to the assumptions~\eqref{ass:modpar3} on the mobility coefficients, $\mu_i$, which descend from assumptions~\eqref{eq:gamma_conditions} when conditions~\eqref{ass:formderi} hold. \textcolor{black}{The above PDEs alongside the constitutive relation~\eqref{eq:pressure_definitionderiv} can then be rewritten as the continuum model~\eqref{eq:PDEmodelrev}.}

\section{\textcolor{black}{Addendum to {\it Step 8} of the proof of Theorem~\ref{th:theo1}}}\label{app:appendixC}
\textcolor{black}{
Differentiating the differential equation~\eqref{eq:TW2_1} with respect to $z$ gives
$$
- c \, \left(p'\right)' - \mu_1 \left[p \, (p')'' + 2 \, p' \, (p')' + p'' \, p'  \right] = \left[\alpha_1 \, \frac{\rmd G}{\rmd p} \, p +  \alpha_1 \, G\left(p\right)\right] \, p',
$$
and, given the conditions~\eqref{eq:TW2_2}, we complement the above differential equation with the following boundary conditions
$$
p'(- \infty) = 0, \qquad p'(0^-) = -\frac{c}{\mu_1}.
$$
}

\textcolor{black}{
On the other hand, differentiating the differential equation~\eqref{eq:TW2_1} with respect to $c$ yields
$$
- c \, \left(\frac{\partial p}{\partial c}\right)' - \mu_1 \left[p \, \left(\frac{\partial p}{\partial c}\right)'' + 2 \, p'  \, \left(\frac{\partial p}{\partial c}\right)' + p'' \, \frac{\partial p}{\partial c}  \right] =  \left[\alpha_1 \, \frac{\rmd G}{\rmd p} \, p +  \alpha_1 \, G\left(p\right)\right] \, \frac{\partial p}{\partial c} + p'
$$
and, given the conditions~\eqref{eq:TW2_2}, we complement the above differential equation with the following boundary conditions
$$
\frac{\partial p}{\partial c}(- \infty) = 0, \qquad \left(\frac{\partial p}{\partial c}\right)'(0^-) = -\frac{1}{\mu_1}.
$$
}

\textcolor{black}{Introducing the notation $f := p'$ and $\displaystyle{g := \frac{\partial p}{\partial c}}$, the above differential equations along with the corresponding boundary conditions can be rewritten, respectively, as 
\begin{subnumcases}{\label{e.mono1bisFBb71}}
- c \, \left(f\right)' - \left[a_1 \, (f)'' + a_2 \, (f)' + a_3 \, f  \right] = a_4 \, f, \quad z \in (-\infty,0), & \label{e.mono1bisFBb71_1}
   \\
   \nonumber\\
f(- \infty) = 0, \qquad f(0^-) = -\dfrac{c}{\mu_1} \label{e.mono1bisFBb71_2}
\end{subnumcases}
and 
\begin{subnumcases}{\label{e.mono2bisFBb61}}
- c \, \left(g\right)' - \left[a_1 \, \left(g\right)'' + a_2 \,  \left(g\right)' + a_3 \, g  \right] =  a_4 \, g + f, \quad z \in (-\infty,0), & \label{e.mono2bisFBb61_1}
   \\
   \nonumber\\
g(- \infty) = 0, \qquad g'(0^-) = -\dfrac{1}{\mu_1}, \label{e.mono2bisFBb61_2}
\end{subnumcases}
where the coefficients $a_1, \ldots, a_4$ are implicitly defined. Since $f<0$ on $(-\infty,0)$ and $g'(0^-)<0$, noting that both $f(- \infty) = 0$ and $g(- \infty) = 0$ and the right-hand side of the differential equation~\eqref{e.mono2bisFBb61_1} contains the additional term $f<0$ compared to the right-hand side of the differential equation~\eqref{e.mono1bisFBb71_1}, we deduce that $\displaystyle{g = \frac{\partial p}{\partial c} < 0}$ on $(-\infty,0)$, from which we conclude that $c \mapsto p(0^-)$ is monotonically decreasing.}

\section{Numerical scheme for the PDE model~\eqref{eq:PDEmodel}}\label{app:appendixB}
To construct numerical solutions of the PDE system~\eqref{eq:PDEmodel} we use a finite-volume scheme with an upwind stabilisation as well as a MUSCL reconstruction of interface values to increase the spatial order of accuracy. This scheme is built upon those we employed in~\cite{bubba2020discrete,lorenzi2023derivation}. The full details are provided below.

We first discretise the 1D domain into spatial cells of size $ \Delta x$.
The time discretisation is a forward Euler method. We allow the time-step to vary during the simulation using the CFL-like condition
\[
 \frac{\Delta t}{\Delta x} \max(\mu_1,\ldots,\mu_I) \max_x( |\nabla p|) < 1.
\]

For each spatial cell interface, we reconstruct the values of the cell densities adopting a MUSCL approach. For the $i^{th}$ component of the solution we have
\[ 
 n^L_{i(j+1/2)} = n^k_{i(j)} + 0.5 \phi(r_{j})(n^k_{i(j+1)}-n^k_{i(j)}), \quad n^R_{i(j+1/2)} = n^k_{i(j)} - 0.5 \phi(r_{j+1})(n^k_{i(j+2)}-n^k_{i(j+1)}),
\]

\[ 
 n^L_{i(j-1/2)} = n^k_{i(j-1)} + 0.5 \phi(r_{j-1})(n^k_{i(j)}-n^k_{i(j-1)}), \quad n^R_{i(j-1/2)} = n^k_{i(j)} - 0.5 \phi(r_{j})(n^k_{i(j+1)}-n^k_{i(j)}),
\]
where
\[
r_j=\frac{n^k_{i(j)}-n^k_{i(j-1)}}{n^k_{i(j+1)}-n^k_{i(j)}},
\]
and the function $\phi(\cdot)$ is the chosen flux-limiter. In particular, here we choose the \textit{ospre} flux-limiter, i.e.
\[
\phi(r)=\frac{1.5 \left(r^2+r\right)}{\left(r^2+r+1\right)}.
\]
We then compute the gradient of the pressure at each cell interface. {\textcolor{black}{To do so, we first define 
\[
p_{j+1/2}^L=\sum_{i=1}^I \omega_i n_{i(j+1/2)}^L, \quad p_{j+1/2}^R=\sum_{i=1}^I \omega_i n_{i(j+1/2)}^R.
\]
We then compute the quantities $\nabla p_{j+1/2}^L$ and $\nabla p_{j+1/2}^R$ using: a central differencing scheme in the interior of the spatial domain; a forward  differencing scheme at the left endpoint of the spatial domain; and a backward differencing scheme at the right endpoint of the spatial domain. Finally, we employ}} the reconstruction
\[
 (\nabla p)_{j+1/2}= 0.5((\nabla p)_{j+1/2}^L+ (\nabla p)_{j+1/2}^R ).
\]
The flux across the spatial cell interface for the discretisation of the $i^{th}$ component of the solution is given by 
\[
F_{j+1/2}^{n_i} = \mu_i \, n^L_{i(j+1/2)} \max(0, -(\nabla p)_{j+1/2}) + \mu_i\, n^R_{i(j+1/2)} \min(0, -(\nabla p)_{j+1/2}). 
\]

Altogether, upon splitting between conservative and non-conservative terms, the numerical scheme reads as
\begin{eqnarray*}
 n^{k+1/2}_{i(j)} &=& n^k_{i(j)} + \Delta t \left[- \frac{F^{n_i}_{j+1/2}-F^{n_i}_{j-1/2}}{\Delta x} \right],\\
 n^{k+1}_{i(j)} &=& n^{k+1/2}_{i(j)} + \Delta t \, \alpha_i \, n^{k+1/2}_{i(j)} G(p_{j}^k).
\end{eqnarray*}

\end{appendices}

\bibliography{sn-bibliography}

\end{document}